\documentclass[review,12pt]{elsarticle}
\usepackage{amsmath}
\usepackage{amsfonts}
\usepackage{geometry}   
\usepackage[parfill]{parskip}    
\usepackage{graphicx}
\usepackage{amssymb}
\usepackage{epstopdf}
\usepackage{psfrag}
\newcommand{\figref}[1]{{Figure~\ref{#1}}}
\newcommand {\bea}{\begin{eqnarray}}
\newcommand {\ea}{\end{eqnarray}}
\newtheorem{proposition}{Proposition}[section]
\newtheorem{theorem}{Theorem}[section]
\newtheorem{Assumption}{Assumption}[section]
\newtheorem{lemma}{Lemma}[section]
\newtheorem{remark}{Remark}[section]

\newenvironment{proof}[1][Proof]{\textbf{#1.} }{\hspace{\stretch{1}}\rule{0.5em}{0.5em}}

\usepackage{xspace}

\newcommand{\Dt}{\Delta t}
\usepackage{subfigure}

\newcommand{\thmref}[1]{{Theorem~\ref{#1}}}
\newcommand{\lemref}[1]{{Lemma~\ref{#1}}}

\newcommand{\assref}[1]{{Assumption~\ref{#1}}}
\newcommand{\propref}[1]{{Proposition~\ref{#1}}}

\journal{Applied Mathematics and Computation}
\begin{document}
\begin{frontmatter}
\title{Strong convergence of the linear implicit Euler method for the finite element discretization of  semilinear   SPDEs driven by  multiplicative  or additive noise}

\author[ata,atb,atc]{Antoine Tambue}
\cortext[cor1]{Corresponding author}
\ead{antonio@aims.ac.za}
\address[ata]{Department of Computing Mathematics and Physics,  Western Norway University of Applied Sciences, Inndalsveien 28, 5063 Bergen, Norway.}
\address[atb]{The African Institute for Mathematical Sciences(AIMS) of South Africa and Stellenbosh University,
6-8 Melrose Road, Muizenberg 7945, South Africa.}
\address[atc]{Center for Research in Computational and Applied Mechanics (CERECAM), and Department of Mathematics and Applied Mathematics, University of Cape Town, 7701 Rondebosch, South Africa.}

\author[jdm]{Jean Daniel Mukam}
\ead{jean.d.mukam@aims-senegal.org}
\address[jdm]{Fakult\"{a}t f\"{u}r Mathematik, Technische Universit\"{a}t Chemnitz, 09126 Chemnitz, Germany.}

%

\begin{abstract}
This paper aims to investigate the numerical approximation of a general second order  parabolic stochastic partial differential equation(SPDE) 
driven by multiplicative or additive noise. The linear operator is not necessary self-adjoint, so more useful in concrete applications. 
The SPDE is discretized in space  by the finite element method and in time  by the  linear implicit Euler method.  The corresponding scheme is more stable and efficient to solve stochastic advection-dominated reactive transport in porous media. This extends the current results in the literature to not necessary self-adjoint operator. 
As a challenge,  key part of the proof  does not rely anymore on the spectral decomposition of the linear operator. 
The results reveal how the  convergence orders  depend on the regularity of the noise and the initial data.  
In particular for multiplicative trace class noise, we achieve optimal  convergence order $\mathcal{O}(h^2+\Delta t^{1/2})$ 
and for additive trace class noise, we achieve   optimal convergence order in space and sup-optimal convergence order in time  of the form $\mathcal{O}(h^2+\Delta t^{1-\epsilon})$, 
for an arbitrarily small $\epsilon>0$. 
Numerical experiments to sustain our theoretical results are provided.
\end{abstract}

\end{frontmatter}
\section{Introduction}
\label{intro}
We consider numerical approximation of SPDE defined in $\Lambda\subset \mathbb{R}^d$, $d=1,2,3$,  with initial value and  boundary conditions (Dirichlet, Neumann  or Robin boundary conditions). We consider the parabolic SPDE of the form
\begin{eqnarray}
\label{model}
dX(t)+AX(t)dt=F(X(t))dt+B(X(t))dW(t), \quad X(0)=X_0, \quad t\in(0,T]
\end{eqnarray}
 on the Hilbert space $L^2(\Lambda)$. We denote by $T>0$ the final time, $F$ and $B$ are nonlinear functions,  $X_0$ is the initial data which  is random,
$A$ is a linear operator, unbounded, not necessarily self-adjoint, and $-A$ is assumed to be a generator of an analytic semigroup $S(t)=:e^{-t A}$,  $t\geq 0.$
The noise  $W(t)=W(x,t)$ is a $Q$-Wiener process defined in a filtered probability space $\left(\Omega,\mathcal{F}, \mathbb{P}, \{\mathcal{F}_t\}_{t\geq 0}\right)$.
The filtration is assumed to fulfill the usual conditions (see \cite[Definition 2.1.11]{Prevot}). It is well known  that the noise can be represented as follows
\begin{eqnarray}
\label{noise}
W(x,t)=\sum_{i\in \mathbb{N}^d}\sqrt{q_i}e_i(x)\beta_i(t), \quad t\in[0,T],
\end{eqnarray} 
where $q_i$, $e_i$, $i\in\mathbb{N}^d$ are respectively the eigenvalues and the eigenfunctions of the covariance operator $Q$,
and $\beta_i$ are independent and identically distributed standard Brownian motions. 
Precise assumptions on $F$, $B$, $X_0$ and $A$  will be given in the next section 
to ensure the existence of the unique mild solution $X$ of \eqref{model}, which has the following representation (see \cite{Prato,Prevot}) 
\begin{eqnarray}
\label{mild1}
X(t)=S(t)X_0+\int_0^tS(t-s)F(X(s))ds+\int_0^tS(t-s)B(X(s))dW(s),\quad t\in(0,T].
\end{eqnarray}
Equations of type \eqref{model} are used to model different real world phenomena in different fields such as biology, chemistry, physics etc \cite{Shardlow,ATthesis,SebaGatam}. In more cases 
explicit solutions of SPDEs are unknown, therefore numerical methods are powerful tools appropriated for their approximations. Numerical approximation of SPDE of type \eqref{model}  
is therefore an active research area and have attracted a lot of attentions since two decades, see e.g., \cite{Antonio1,Xiaojie1,Xiaojie2,Raphael,Jentzen1,Jentzen2,Yan1,Yan2,Shardlow,Printems,Kovac1}
and references therein. Due to the time step restriction of the explicit Euler method, linear implicit Euler method is used in many situations.
Linear implicit Euler method have been investigated in the literature, see e.g., \cite{Raphael,Xiaojie2,AntGaby}. 
The work in \cite{AntGaby} deals with the case of additive noise with self-ajoint operator and uses the spectral Galerkin method for the space discretization while 
the work in \cite{Raphael} still deals with self-adjoint operator,  multiplicative noise  and uses the standard finite element method for space discretization. 
The work in \cite{Xiaojie2} considers the case of additive noise with self-adjoint operator and also uses  finite element method for space dicretization.
Note that the proofs of the results  in \cite{Raphael,Xiaojie2,AntGaby} are heavily based on the spectral decomposition of the unbounded linear operator $A$, see e.g., \cite[Lemma 4.4, Lemma 4.3, Lemma 4.2 \& Lemma 4.1]{Raphael}. In fact, the proof of \cite[Lemma 4.1]{Raphael} and \cite[Lemma 4.2]{Raphael} use respectively \cite[Theorem 3.4]{Vidar} and \cite[Theorem 3.3]{Vidar}, which are based on the self-adjoint assumption \cite[(i), Page 40, Chapter 3]{Vidar}.  Note also that that the proof of \cite[Lemma 4.3 (i)-(ii)]{Raphael}   rely on \cite[Theorem 7.7 \& Theorem 7.8]{Vidar}, which is based on the spectral decomposition. More precisely,  the proof of \cite[Theorem 7.7 \& Theorem 7.8]{Vidar} are based on \cite[Theorem 7.2 \& Theorem 7.3]{Vidar}, which use the spectral decomposition. The proof of \cite[Lemma 4.3 (iii)]{Raphael} also uses the spectral decomposition.  The proof of the results in \cite{Xiaojie2} are also strongly based on \cite[Theorem 7.7 \& Theorem 7.8]{Vidar} and some preparatory results for self adjoint operator in \cite{Raphael}. Therefore results in \cite{Raphael, Xiaojie2,AntGaby} cannot be easily extended to the case of non self-adjoint operator. Our aim in this work is  to fill that gap and  investigate the case of not necessary self-adjoint operator, 
more useful in concrete applications, which has not yet been investigated in the literature to best of our knowledge. Note that although the work in \cite{Raphael} 
solves general second order stochastic parabolic PDEs
considered  here by adding the advection term on nonlinear function $F$\footnote{This technique has failed in our numerical example provided in Section \ref{numexperiment}}, the linear implicit Euler method in such approach behaves as the unstable explicit Euler method 
for strong advection term. An illustrative example is the stochastic dominated  transport flow in  porous media  with high P\'{e}clet number \cite{SebaGatam}. 
In such cases, to stablilize the implicit scheme, it is important to  include the advection term in the linear operator, which is treated implicitly in the linear implicit method,
but current works in   \cite{Raphael,Xiaojie2,AntGaby} are  not longer applicable since the linear operator is not longer self-adjoint.
Our goal here is to fill  this gap and provide strong convergence results for multiplicative and additive noise for SPDE with not necessary self-adjoint operator. Note that numerical schemes for problem \eqref{model} with not necessary self-adjoint operator $A$  have been performed in \cite{Antonio1,Antjd1} for exponential integrators. Since solving linear systems are more straightforward than computing the exponential of matrix, it is important to develop alternative methods based on the resolution of linear systems, which may be more efficient if the appropriate preconditionners are used. The analysis of the linear implicit Euler scheme with not necessary self-adjoint operator is more difficult than those of an exponential schemes since the errors estimates in the corresponding determistic linear problem (Section \ref{erreurdeter}) are extremely complex for non self-adjoint operator. Keys part of this work are based on those errors estimates with elegant properties of the semi-group and its approximation, in opposite of the spectral decomposition currently used in the literature \cite{Raphael,Xiaojie2}. 
The results indicate how the convergence orders depend on the regularity 
of the initial data and the noise. In particular, we achieve the optimal convergence orders $\mathcal{O}\left(h^{\beta}+\Delta t^{\min(\beta,1)/2}\right)$
for multiplicative noise;  the  optimal convergence order in space and sup-optimal convergence order in time of the form $\mathcal{O}\left(h^{\beta}+\Delta t^{\min({\beta/2}, 1-\epsilon)}\right)$ 
for additive noise, where $\beta$ is the regularity's parameter of the noise (see \assref{assumption2}) and $\epsilon>0$ is an arbitrarily real number small enough.
It is worth to mention that these optimal convergence orders were also achieved in \cite{Raphael,Xiaojie2} with self-adjoint operator,
where the convergence analysis was heavily based on the spectral decomposition of the linear operator $A$.

The rest of this paper is organized as follows.  Section \ref{wellposed}
deals with the well posedness problem, the numerical scheme and the main results. In section \ref{convergenceproof}, 
we first  provide some error estimates for the deterministic homogeneous problem as preparatory results and we  provide
the proof of the main results. Section \ref{numexperiment} provides some numerical experiments to sustain the theoretical results.

\section{Mathematical setting and main results}
\label{wellposed}
\subsection{Setting and assumptions}
\label{notation}
Let us define functional spaces, norms and notations that  will be used in the rest of the paper. 
Let  $(H,\langle.,.\rangle_H,\Vert .\Vert)$ be a separable Hilbert space.  For  $p\geq 2$ and for a Banach space $U$,
we denote by $L^p(\Omega, U)$ the Banach space of all equivalence classes of $p$ integrable $U$-valued random variables. We denote  by $L(U,H)$ 
 the space of bounded linear mappings from $U$ to $H$ endowed with the usual  operator norm $\Vert .\Vert_{L(U,H)}$. By  $\mathcal{L}_2(U,H):=HS(U,H)$, we  denote the space of Hilbert-Schmidt operators from $U$ to $H$. 
  We equip $\mathcal{L}_2(U,H)$ with the norm
 \begin{eqnarray}
 \label{def1}
 \Vert l\Vert^2_{\mathcal{L}_2(U,H)} :=\sum_{i=1}^{\infty}\Vert l\psi_i\Vert^2,\quad l\in \mathcal{L}_2(U,H),
 \end{eqnarray}
 where $(\psi_i)_{i=1}^{\infty}$ is an orthonormal basis of $U$. Note that \eqref{def1} is independent of the orthonormal basis of $U$.
  For simplicity we use the notations $L(U,U)=:L(U)$ and $\mathcal{L}_2(U,U)=:\mathcal{L}_2(U)$. It is well known that for all $l\in L(U,H)$ and $l_1\in\mathcal{L}_2(U)$, $ll_1\in\mathcal{L}_2(U,H)$ and 
  \begin{eqnarray}
  \label{trace1}
  \Vert ll_1\Vert_{\mathcal{L}_2(U,H)}\leq \Vert l\Vert_{L(U,H)}\Vert l_1\Vert_{\mathcal{L}_2(U)}.
  \end{eqnarray}
 We assume that the covariance operator  $Q : H\longrightarrow H$ is positive and self-adjoint. 
 The space of Hilbert-Schmidt operators from  $Q^{1/2}(H)$ to $H$ is denoted by $L^0_2:=\mathcal{L}_2(Q^{1/2}(H),H)=HS(Q^{1/2}(H),H)$  
 with the corresponding norm $\Vert.\Vert_{L^0_2}$  defined by 
\begin{eqnarray}
\label{def2}
\Vert l\Vert_{L^0_2} :=\Vert lQ^{1/2}\Vert_{HS}=\left(\sum_{i=1}^{\infty}\Vert lQ^{1/2}e_i\Vert^2\right)^{1/2}, \quad  l\in L^0_2,
\end{eqnarray} 
where $(e_i)_{i=1}^{\infty}$ is an orthonormal basis  of $H$.
Note that \eqref{def2} is independent of the orthonormal basis of $H$. In the rest of this paper, we take $H=L^2(\Lambda)$.

Throughout this paper, we assume that $\Lambda$ is bounded and has smooth boundary or is a convex polygon of $\mathbb{R}^d$, $d=1,2,3$. In the rest of this paper
we consider the  SPDE \eqref{model} to be of the following form
\begin{eqnarray}
\label{secondorder}
dX(t,x)+[-\nabla \cdot \left(\mathbf{D}\nabla X(t,x)\right)+\mathbf{q} \cdot \nabla X(t,x)]dt=f(x,X(t,x))dt+b(x,X(t,x))dW(t,x),
\end{eqnarray}
$ x\in\Lambda$, $t\in[0,T]$,
where the function $f : \Lambda\times \mathbb{R}\longrightarrow \mathbb{R}$ is continuously twice differentiable and the function
$b : \Lambda\times\mathbb{R}\longrightarrow \mathbb{R}$ is continuously differentiable with globally bounded derivatives.
In the abstract framework \eqref{model}, the linear operator $A$ takes the form
\begin{eqnarray}
\label{operator}
Au=-\sum_{i,j=1}^{d}\dfrac{\partial}{\partial x_i}\left(D_{ij}(x)\dfrac{\partial u}{\partial x_j}\right)+\sum_{i=1}^dq_i(x)\dfrac{\partial u}{\partial x_i},\quad
\mathbf{D}=\left(D_{i,j} \right)_{1\leq i,j \leq d},\,\,\,\,\,\,\, \mathbf{q}=\left( q_i \right)_{1 \leq i \leq d}.
\end{eqnarray}
where $D_{ij}\in L^{\infty}(\Lambda)$, $q_i\in L^{\infty}(\Lambda)$. We assume that there exists a  constant $c_1>0$ such that 
\begin{eqnarray}
\label{ellipticity}
\sum_{i,j=1}^dD_{ij}(x)\xi_i\xi_j\geq c_1|\xi|^2, \quad \forall \xi\in \mathbb{R}^d,\quad x\in\overline{\Omega}.
\end{eqnarray}

The functions $F : H\longrightarrow H$ and $B : H\longrightarrow HS\left(Q^{1/2}(H), H\right)$ are defined by 
\begin{eqnarray}
\label{nemystskii}
\left(F(v)\right)(x)=f\left(x,v(x)\right) \quad \text{and} \quad \left(B(v)u\right)(x)=b\left(x,v(x)\right).u(x),
\end{eqnarray}
for all $x\in \Lambda$, $v\in H$, $u\in Q^{1/2}(H)$, with $H=L^2(\Lambda)$.
 For an appropriate family of eigenfunctions $(e_i)$ such that $\sup\limits_{i\in\mathbb{N}^d}\left[\sup\limits_{x\in \Lambda}\Vert e_i(x)\Vert\right]<\infty$, 
 it is well known \cite[Section 4]{Arnulf1} that the Nemytskii operator $F$ related to $f$ and the multiplication operator $B$ associated 
 to $b$ defined in \eqref{nemystskii} satisfy Assumption \ref{assumption3}, Assumption \ref{assumption4} and Assumption \ref{assumption5}.
As in \cite{Antonio1,Suzuki} we introduce two spaces $\mathbb{H}$ and $V$, such that $\mathbb{H}\subset V$; the two spaces depend on the boundary 
conditions and the domain of the operator $A$. For  Dirichlet (or first-type) boundary conditions we take 
\begin{eqnarray*}
V=\mathbb{H}=H^1_0(\Lambda)=\overline{C^{\infty}_{c}(\Lambda)}^{H^1(\Lambda)}=\{v\in H^1(\Lambda) : v=0\quad \text{on}\quad \partial \Lambda\}.
\end{eqnarray*}
For Robin (third-type) boundary condition and  Neumann (second-type) boundary condition, which is a special case of Robin boundary condition, we take $V=H^1(\Lambda)$
\begin{eqnarray*}
\mathbb{H}=\{v\in H^2(\Lambda) : \partial v/\partial \mathtt{v}_{ A}+\alpha_0v=0,\quad \text{on}\quad \partial \Lambda\}, \quad \alpha_0\in\mathbb{R},
\end{eqnarray*}
where $\partial v/\partial \mathtt{v}_{ A}$ is the normal derivative of $v$ and $\mathtt{v}_{ A}$ is the exterior pointing normal at $n=(n_i)$ to the boundary of $A$, given by
\begin{eqnarray*}
\partial v/\partial\mathtt{v}_{A}=\sum_{i,j=1}^dn_i(x)D_{ij}(x)\dfrac{\partial v}{\partial x_j},\,\,\qquad x \in \partial \Lambda.
\end{eqnarray*}
Using the Green's formula and the boundary conditions, the  corresponding bilinear form associated to $A$  is given by
\begin{eqnarray*}
a(u,v)=\int_{\Lambda}\left(\sum_{i,j=1}^dD_{ij}\dfrac{\partial u}{\partial x_i}\dfrac{\partial v}{\partial x_j}+\sum_{i=1}^dq_i\dfrac{\partial u}{\partial x_i}v\right)dx, \quad u,v\in V,
\end{eqnarray*}
for Dirichlet and Neumann boundary conditions, and  
\begin{eqnarray*}
a(u,v)=\int_{\Lambda}\left(\sum_{i,j=1}^dD_{ij}\dfrac{\partial u}{\partial x_i}\dfrac{\partial v}{\partial x_j}+\sum_{i=1}^dq_i\dfrac{\partial u}{\partial x_i}v\right)dx+\int_{\partial\Lambda}\alpha_0uvdx, \quad u,v\in V,
\end{eqnarray*}
for Robin boundary conditions. Using the G\aa rding's inequality (see e.g. \cite{ATthesis}), it holds that there exist two constants $c_0$ and $\lambda_0\geq 0$ such that
\begin{eqnarray}
\label{ellip1}
a(v,v)\geq \lambda_0\Vert v \Vert^2_{H^1(\Lambda)}-c_0\Vert v\Vert^2, \quad v\in V.
\end{eqnarray}
By adding and substracting $c_0Xdt$ in both sides of \eqref{model}, we have a new linear operator
 still denoted by $A$, and the corresponding  bilinear form is also still denoted by $a$. Therefore, the following coercivity property holds
\begin{eqnarray}
\label{ellip2}
a(v,v)\geq \lambda_0\Vert v\Vert^2_1,\quad v\in V.
\end{eqnarray}
Note that the expression of the nonlinear term $F$ has changed as we included the term $c_0X$ in a new nonlinear term that we still denote by  $F$. The coercivity property (\ref{ellip2}) implies that $-A$ is sectorial in $L^{2}(\Lambda)$, i.e.  there exist $C_{1},\, \theta \in (\frac{1}{2}\pi,\pi)$ such that
\begin{eqnarray*}
 \Vert (\lambda I +A )^{-1} \Vert_{L(L^{2}(\Lambda))} \leq \dfrac{C_{1}}{\vert \lambda \vert },\;\quad \quad
\lambda \in S_{\theta},
\end{eqnarray*}
where $S_{\theta}=\left\lbrace  \lambda \in \mathbb{C} :  \lambda=\rho e^{i \phi},\; \rho>0,\;0\leq \vert \phi\vert \leq \theta \right\rbrace $ (see \cite{Henry}).
 Then  $-A$ is the infinitesimal generator of a bounded analytic semigroup $S(t)=e^{-t A}$  on $L^{2}(\Lambda)$  such that
\begin{eqnarray*}
S(t)= e^{-t A}=\dfrac{1}{2 \pi i}\int_{\mathcal{C}} e^{ t\lambda}(\lambda I +A)^{-1}d \lambda,\;\;\;\;\;\;\;
\;t>0,
\end{eqnarray*}
where $\mathcal{C}$  denotes a path that surrounds the spectrum of $-A $.
The coercivity  property \eqref{ellip2} also implies that $A$ is a positive operator and its fractional powers are well defined  for any $\alpha>0,$ by
\begin{equation}
\label{fractional}
 \left\{\begin{array}{rcl}
         A^{-\alpha} & =& \frac{1}{\Gamma(\alpha)}\displaystyle\int_0^\infty  t^{\alpha-1}{\rm e}^{-tA}dt,\\
         A^{\alpha} & = & (A^{-\alpha})^{-1},
        \end{array}\right.
\end{equation}
where $\Gamma(\alpha)$ is the Gamma function (see \cite{Henry}).

Under condition \eqref{ellipticity}, it is well known (see e.g. \cite{Suzuki}) that the linear operator $-A$ given by \eqref{operator} generates an analytic semigroup $S(t)\equiv e^{-tA}$. 
Following \cite{Larsson1,Larsson2,Antonio1,Suzuki}, we characterize the domain of the operator $A^{r/2}$ denoted by $\mathcal{D}(A^{r/2})$, $r\in\{1,2\}$ with the following equivalence of norms, useful in our convergence proofs
\begin{eqnarray*}
\Vert v\Vert_{H^1(\Lambda)}\equiv \Vert A^{r/2}v\Vert=:\Vert v\Vert_r,\quad  v\in\mathcal{D}(A^{r/2}),\\
\mathcal{D}(A^{r/2})=\mathbb{H}\cap H^{r}(\Lambda), \quad \text{ (for Dirichlet boundary conditions)},\\
\mathcal{D}(A)=\mathbb{H}, \quad \mathcal{D}(A^{1/2})=H^1(\Lambda), \quad \text{(for Robin boundary conditions)}.
\end{eqnarray*}

\begin{remark}
 For Dirichlet boundary conditions, if $\bigtriangledown \cdot \mathbf{q}=0 $,  as we can see in \cite{Antonio3}, the coercivity \eqref{ellip2} is obtained with $c_0=0$, 
 and  G\aa rding's inequality
 is not needed.  We can aslo notice from \cite{Antonio3} that the constant $c_0$ depends only on $\bigtriangledown \cdot \mathbf{q}$. 
 Indeed for  $\mathbf{q}$ smooth enough, $\mathbf{q} \cdot \nabla X = \nabla \cdot (\mathbf{q}X)-(\bigtriangledown \cdot \mathbf{q}) X$, 
 so adding  and subtracting $c_0Xdt$ in both sides of \eqref{model}, 
 should  in many realistic cases always weaken the advection term even for strong divergence flow  ($c_0$ large), and therefore 
 stabilize the semi implicit scheme  more than  the work in \cite{Raphael} where the whole advection term is kept in the nonlinear function $F$.
 \end{remark}

In order to ensure the existence and the uniqueness of solution for SPDE \eqref{model} and for the purpose of the convergence analysis, we make the following assumptions.
\begin{Assumption}\textbf{[Initial value $X_0$]}
\label{assumption2} Let $p\in[2,\infty)$. 
We assume that the initial data $X_0$ belongs to  $L^p(\Omega, \mathcal{D}((-A)^{\beta/2}))$, $0< \beta\leq 2$.
\end{Assumption}
\begin{Assumption}\textbf{[Nonlinear term $F$]}
\label{assumption3} 
The nonlinear function $F: H\longrightarrow H$ is   Lipschitz continuous, i.e. there exists a constant $C>0$ such that 
\begin{eqnarray}
\label{reviewinequal1}
\Vert F(0)\Vert\leq C,\quad \Vert F(Y)-F(Z)\Vert \leq C\Vert Y-Z\Vert, \quad Y,Z\in H.
\end{eqnarray}
\end{Assumption}
As a consequence of \eqref{reviewinequal1}, it holds that
\begin{eqnarray*}
\Vert F(Z)\Vert\leq C (1+\Vert Z\Vert), \quad Z\in H.
\end{eqnarray*}
Following \cite[Chapter 7]{Prato} or \cite{Yan2,Antonio1,Arnulf1,Raphael}, we make the following assumption on the diffusion term.
 \begin{Assumption}\textbf{[Diffusion term ]} 
 \label{assumption4}
  We assume that the operator  $B : H \longrightarrow L^0_2$ satisfies the global Lipschitz condition, i.e. there exists a positive constant $C$ such that 
 \begin{eqnarray*}
 \Vert B(0)\Vert_{L^0_2}\leq C,\quad \Vert B(Y)-B(Z)\Vert_{L_2^0}\leq C\Vert Y-Z\Vert, \quad Y,Z\in H.
 \end{eqnarray*}
 As a consequence,  it holds that 
 \begin{eqnarray*}
 \Vert B(Z)\Vert_{L^0_2}\leq L(1+\Vert Z\Vert), \quad  Z\in H.
 \end{eqnarray*}
 \end{Assumption}

We equip $V_{\alpha}:=\mathcal{D}(A^{\alpha/2})$, $\alpha\in \mathbb{R}$ with the norm  $\Vert v\Vert_{\alpha}:=\Vert A^{\alpha/2}v\Vert$, for all $v\in H$. 
It is well known that $(V_{\alpha}, \Vert .\Vert_{\alpha})$ is a Banach space \cite{Henry}.

To establish  our $L^p$ strong convergence result when dealing with multiplicative noise, we will also need the following further assumption on the diffusion term when $\beta \in [1,2]$, which was also used in  \cite{Arnulf1,Stig1} to achieve optimal order regularity and in \cite{Antonio1,Raphael,Antjd1} to achieve optimal convergence order in space and time.
\begin{Assumption}
\label{assumption5}
There exists a positive constant $c>0$ such that $B\left(\mathcal{D}\left(A^{(\beta-1)/2}\right)\right)\subset HS\left(Q^{1/2}(H),\mathcal{D}\left(A^{(\beta-1)/2}\right)\right)$ and 
$\left\Vert A^{(\beta-1)/2}B(v)\right\Vert_{L^0_2}\leq c\left(1+\Vert v\Vert_{\beta-1}\right)$ for all  $v\in\mathcal{D}\left(A^{(\beta-1)/2}\right)$, where $\beta$ comes from \assref{assumption2}.
\end{Assumption}
Typical examples which fulfill \assref{assumption5} are stochastic reaction diffusion equations (see \cite[Section 4]{Arnulf1}).

 When dealing with additive noise (i.e. when the nonlinear function  $B$ is independent of $X$), the strong convergence proof will make use of the following assumption, also used in \cite{Xiaojie2,Xiaojie1,Antjd1}.
 \begin{Assumption}
 \label{assumption6a}
  We assume that the covariance operator $Q$ satisfy the following estimate
 \begin{eqnarray*}
 \left\Vert A^{\frac{\beta-1}{2}}Q^{\frac{1}{2}}\right\Vert_{\mathcal{L}_2(H)}<C, 
 \end{eqnarray*}
 where $\beta$ comes from \assref{assumption2}.
 \end{Assumption}
 
 When dealing with additive noise, to achieve higher order convergence in time, we assume the nonlinear function to satisfy the following assumption, also used in \cite{Xiaojie1,Xiaojie2,Antjd1}.
 \begin{Assumption}
 \label{assumption6b}
 The deterministic mapping $F: H\longrightarrow H$ is twice differentiable and there exist constants $L\geq 0$ and $\eta\in[0,1)$ such that
 \begin{eqnarray}
 \Vert F'(u)v\Vert\leq L\Vert v\Vert,\quad
 \Vert F''(u)(v_1,v_2)\Vert_{-\eta}\leq L\Vert v_1\Vert.\Vert v_2\Vert,\quad u,v,v_1,v_2\in H.
 \end{eqnarray}
 \end{Assumption}
Let us recall the following proposition which provides some  smooth properties of the semigroup $S(t)$ generated by $-A$, that will be useful in the rest of the paper.  
 \begin{proposition}\textbf{[Smoothing properties of the semigroup]}\cite{Henry}
 \label{theorem1}
 \label{prop1} Let $\alpha > 0$, $\delta\geq 0$  and $0\leq \gamma\leq 1$, then there exists a constant $C>0$ such that 
 \begin{eqnarray*}
 \Vert A^{\delta}S(t)\Vert_{L(H)}\leq Ct^{-\delta}, \quad t>0,\quad  
 \Vert A^{-\gamma}\left(\mathbf{I}-S(t)\right)\Vert_{L(H)}\leq Ct^{\gamma}, \quad t\geq 0\\
 A^{\delta}S(t)=S(t)A^{\delta}, \quad \text{on} \quad \mathcal{D}(A^{\delta}),\quad
 \Vert D^l_tS(t)v\Vert_{\delta}\leq Ct^{-l-(\delta-\alpha)/2}\Vert v\Vert_{\alpha},\quad t>0,\quad v\in\mathcal{D}(A^{\alpha}),
 \end{eqnarray*}
 where $l=0,1$, 
 and  $D^l_t=\dfrac{d^l}{dt^l}$.
 If $\delta\geq \gamma$ then  $\mathcal{D}(A^{\delta})\subset \mathcal{D}(A^{\gamma})$.
 \end{proposition}

\begin{theorem}\cite[Theorem 7.2]{Prato}\\
\label{theorem2}
Let   \assref{assumption3} and \assref{assumption4} be satisfied. If $X_0$ is a $\mathcal{F}_0$- measurable $H$ valued random variable, 
then there exists a unique mild solution $X$ of the problem \eqref{model} represented by \eqref{mild1} and  satisfying 
\begin{eqnarray*}
\mathbb{P}\left[\int_0^T\Vert X(s)\Vert^2ds<\infty\right]=1,
\end{eqnarray*}
and for any $p\geq 2$, there exists a constant $C=C(p,T)>0$ such that 
\begin{eqnarray*}
\sup_{t\in[0,T]}\mathbb{E}\Vert X(t)\Vert^p\leq C(1+\mathbb{E}\Vert X_0\Vert^p).
\end{eqnarray*}
\end{theorem}

\subsection{Fully discrete scheme and main results}
\label{strongresult}
Let us start with the space  discretization of our problem \eqref{model}.  We start by  splitting  the domain $\Lambda$ in finite triangles.
Let $\mathcal{T}_h$ be the triangulation with maximal length $h$ satisfying the usual regularity assumptions, and  $V_h \subset V$ be the space of continuous functions that are 
piecewise linear over the triangulation $\mathcal{T}_h$. 
We consider the projection $P_h$ from $H=L^2(\Lambda)$ to $V_h$ defined by 
\begin{eqnarray}
\label{projection}
\langle P_hu,\chi\rangle_H=\langle u,\chi\rangle_H, \quad u\in H, \quad  \chi\in V_h.
\end{eqnarray}
The discrete operator $A_h : V_h\longrightarrow V_h$ is defined by 
\begin{eqnarray}
\label{discreteoperator}
\langle A_h\phi,\chi\rangle_H=\langle A\phi,\chi\rangle_H=a(\phi,\chi),\quad  \phi,\chi\in V_h,
\end{eqnarray}
Like $-A$, $-A_h$ is also a generator of an analytic semigroup $S_h(t)=:e^{-tA_h}$, see e.g. \cite{Larsson2,Suzuki}. 
As any analytic semigroup and their generator, $-A_h$ and $S_h(t)$ satisfy the smoothing properties of Proposition  \ref{theorem1}  with a uniform constant $C$ (i.e. independent of $h$), see e.g. \cite{Larsson2,Suzuki}.   
The semi-discrete  version of problem \eqref{model} consists of finding $X^h(t)\in V_h$,  $t\in(0,T]$ such that 
\begin{eqnarray}
\label{semi1}
dX^h(t)+A_hX^h(t)dt=P_hF(X^h(t))dt+P_hB(X^h(t))dW(t),\quad X^h(0)=P_hX_0, \quad t\in(0,T].
\end{eqnarray}
Applying the linear implict Euler method  to \eqref{semi1} gives the following fully discrete scheme 
\begin{eqnarray}
\label{implicit1}
\left\{\begin{array}{ll}
X^h_0=P_hX_0, \quad \\
X^h_{m+1}=S_{h,\Delta t}X^h_m+\Delta tS_{h,\Delta t}P_hF(X^h_m)+S_{h,\Delta t}P_hB(X^h_m)\Delta W_m, 
\end{array}
\right.
\end{eqnarray}
where $\Delta W_m$ and $S_{h,\Delta t}$ are defined respectively by 
\begin{eqnarray}
\label{operator1}
\Delta W_m :=W_{t_{m+1}}-W_{t_{m}}\quad  \text{and}\quad  S_{h,\Delta t}:=(\mathbf{I}+\Delta tA_h)^{-1}.
\end{eqnarray}
Having the numerical method  \eqref{implicit1}  in hand, our goal is to analyze its strong convergence toward the exact solution in the  $L^p$ norm for multiplicative and additive noise.

Throughout this paper we take $t_m=m\Delta t\in[0,T]$, where $T=M\Delta t$ for $m, M\in\mathbb{N}$, $m\leq M$, $T$ is fixed, $C$ is a generic constant that may change from one place to another. 
The main results of this paper are formulated in the following theorems. 
\begin{theorem}
\label{mainresult1}
Let $X(t_m)$ and  $X^h_m$ be respectively the mild solution given by \eqref{mild1} and  the numerical approximation given by \eqref{implicit1} at $t_m=m\Delta t$. 
Let Assumptions  \ref{assumption2},  \ref{assumption3} and \ref{assumption4} be fulfilled.
\begin{itemize}
\item[(i)] If  $0<\beta<1$, then the following error  estimate holds
\begin{eqnarray*}
\Vert X(t_m)-X^h_m\Vert_{L^p(\Omega,H)}\leq C\left(h^{\beta}+\Delta t^{\beta/2}\right).
\end{eqnarray*}
\item[(ii)] If $1\leq\beta\leq 2$ and if \assref{assumption5} is fulfilled, then the following error estimate holds
\begin{eqnarray*}
\Vert X(t_m)-X^h_m\Vert_{L^p(\Omega,H)}\leq C\left(h^{\beta}+\Delta t^{1/2}\right).
\end{eqnarray*}
\end{itemize}
\end{theorem}
For additive noise,  we have the following convergence result.
\begin{theorem}
\label{mainresult2}
In the case of  additive noise, if Assumptions \ref{assumption2}, \ref{assumption3}, \ref{assumption6a} with $\beta\in[0,2)$  and \assref{assumption6b} are fulfilled, then the following error estimate holds for the mild solution $X(t)$ of \eqref{model} and the numerical approximation  \eqref{implicit1}
\begin{eqnarray*}
\Vert X(t_m)-X^h_m\Vert_{L^p(\Omega,H)}\leq C\left(h^{\beta}+\Delta t^{\beta/2}\right).
\end{eqnarray*}
Moreover, if Assumptions \ref{assumption6a} and \ref{assumption2} are fulfilled with $\beta=2$, then the following  holds
\begin{eqnarray*}
\Vert X(t_m)-X^h_m\Vert_{L^p(\Omega,H)}\leq C\left(h^{\beta}+\Delta t^{1-\epsilon}\right),
\end{eqnarray*}
for an arbitrarily small $\epsilon>0$.
\end{theorem}


\section{Proof of the main results}
\label{convergenceproof}
The proof of the main results requires some preparatory results.

\subsection{Errors estimates for deterministic problem}
\label{erreurdeter}
This section extends  results in \cite[Section 4]{Raphael},  \cite[Section 3]{Xiaojie2} and some results in \cite{Vidar} 
to the case of not necessary self-adjoint operator. As we  are dealing with not necessary self-adjoint operator, 
the proof of our estimates does not make use of the spectral decomposition of the linear operator $A$.  
Let us start by introducing the Ritz representation operator $R_h : V\longrightarrow V_h$ defined by
\begin{eqnarray}
\label{ritz}
\langle AR_hv,\chi\rangle_H=\langle Av,\chi\rangle_H=a(v,\chi),\quad  v\in V,\quad  \chi\in V_h.
\end{eqnarray}
 Under the regularity assumptions on the triangulation and in view of the $V$-ellipticity \eqref{ellipticity}, it is well known (see e.g. \cite{Suzuki,Larsson2}) that for all $r\in\{1,2\}$ the following error estimates hold 
\begin{eqnarray}
\label{ritz1}
\Vert R_hv-v\Vert+h\Vert R_hv-v\Vert_{H^1(\Lambda)}\leq Ch^r\Vert v\Vert_{H^r(\Lambda)},\quad v\in V\cap H^r(\Lambda).
\end{eqnarray}

Let us consider the following deterministic linear problem :  Find $u\in V$ such that 
\begin{eqnarray}
\label{model1a}
\dfrac{du}{dt}+Au=0,\quad u(0)=v,\quad t\in (0,T].
\end{eqnarray}
The corresponding semi-discrete problem in space consists  of finding $u_h\in V_h$ such that 
\begin{eqnarray}
\label{deter1}
\dfrac{du_h}{dt}+A_hu_h=0,\quad u_h(0)=P_hv,\quad t\in (0,T].
\end{eqnarray}
Let us define the following operator 
\begin{eqnarray*}
G_h(t):=S(t)-S_h(t)P_h=e^{-At}-e^{-A_ht}P_h,
\end{eqnarray*}
so that $u(t)-u_h(t)=G_h(t)v$. The estimate \eqref{ritz1} was used in \cite{Antonio1,Antjd1} to prove   the following lemma, with the  prove available in \cite[Lemma 3.1]{Antonio1} or \cite[Lemma 7]{Antjd1}.

\begin{lemma}  
\label{lemma1}
 Under the setting in Section \ref{notation}, the following estimate holds
\begin{eqnarray}
\label{addis1}
\Vert G_h(t)v\Vert \leq Ch^{r}t^{-(r-\alpha)/2}\Vert v\Vert_{\alpha},\quad r\in[0,2],\quad \alpha\leq r,\,\, t\in (0,T].
\end{eqnarray}
\end{lemma}

\begin{lemma} 
\label{lemma2}
\begin{itemize}
\item[(i)]
Let $0\leq\rho\leq1$. Then there exists a constant $C$ such that 
\begin{eqnarray}
\label{mil1}
\Vert G_{h}(t)v\Vert\leq Ct^{-\rho/2}\Vert v\Vert_{-\rho}, \quad v\in\mathcal{D}(A^{-\rho}),\quad t>0.
\end{eqnarray}
\item[(ii)] Let $r\in[1,2]$, $1\leq\alpha\leq r$ and $v\in\mathcal{D}(A^{\alpha/2})$. Then there exists a  constant $C\geq0$ such that
\begin{eqnarray*}
\Vert D_t(S(t)-S_h(t)P_h)v\Vert\leq Ch^{\alpha}t^{-1}\Vert v\Vert_{\alpha}+Ch^{r}t^{-1-(r-\alpha)/2}\Vert v\Vert_{\alpha},
\end{eqnarray*}
 for all $t>0$, where $D_t:=\frac{d}{dt}$.
\item[(iii)]
Let $0\leq\rho\leq1$. Then there exists a constant $C$ such that 
\begin{eqnarray}
\label{mil3}
\Vert G_{h}(t)v\Vert\leq Ch^{2-\rho}t^{-1}\Vert v\Vert_{-\rho}, \quad v\in\mathcal{D}(A^{-\rho}),\quad t>0.
\end{eqnarray}
\item[(iv)]
Let $0\leq\rho\leq1$. Then there exists a constant $C$ such that 
\begin{eqnarray}
\label{mil4}
 \left\Vert\int_0^t G_{h}(s)vds\right\Vert\leq Ch^{2-\rho}\Vert v\Vert_{-\rho}, \quad v\in\mathcal{D}(A^{-\rho}),\quad t>0.
\end{eqnarray}
\item[(v)] Let $0\leq\gamma\leq 2$. Then the following estimate holds
\begin{eqnarray*}
\left(\int_0^t\Vert G_h(s)v\Vert^2ds\right)^{1/2}\leq Ch^{\gamma}\Vert v\Vert_{\gamma-1},\quad v\in\mathcal{D}(A^{\gamma-1}),\quad t>0.\nonumber
\end{eqnarray*}
\end{itemize}
\end{lemma}

\begin{proof}
\begin{itemize}
\item[(i)]
The proof of (i) is similar to \cite[Lemma 4.1 (ii)]{Raphael} by using \lemref{lemma2} instead of \cite[Lemma 4.1 (i)]{Raphael}. 

\item[(ii)]
As in \cite{Antonio1}, we write
\begin{eqnarray}
\label{essai2b}
u_h(t)-u(t)=(u_h(t)-R_hu(t))+(R_hu(t)-u(t))=: \theta(t)+\rho(t).
\end{eqnarray}
Therefore the following estimate holds
\begin{eqnarray}
\label{mil2}
\Vert D_t(S(t)-S_h(t)P_h)v\Vert\leq \Vert D_t\theta(t)\Vert+\Vert D_t\rho(t)\Vert.
\end{eqnarray}
From  \cite{Antonio1}, we have
\begin{eqnarray}
\label{essai3}
 A_hR_h=P_hA, \quad \text{and} \quad \theta_t=A_h\theta(t)-P_hD_t\rho(t).
\end{eqnarray}
Taking the derivative in both sides of the second expression of  \eqref{essai3} and using the fact that $D_tP_h=P_hD_t$ (see e.g. \cite[Page 16]{Jens}), we obtain 
\begin{eqnarray}
\label{essai4}
\theta_{tt}=A_h\theta_t-P_hD_{tt}\rho(t).
\end{eqnarray}
Therefore, by Duhamel's principle we have 
\begin{eqnarray}
\label{essai5}
\theta_t(t)=S_h(t)\theta_t(0)-\int_0^tS_h(t-s)P_hD_{ss}\rho(s)ds.
\end{eqnarray}
Splitting the integral part of \eqref{essai5} into two intervals and integrating by parts over the first interval yields
\begin{eqnarray}
\label{essai6}
\theta_t(t)&=&S_h(t)\theta_t(0)+S_h(t)P_hD_t\rho(0)-S_h(t/2)P_hD_t\rho(t/2)\nonumber\\
&+&\int_0^{t/2}(D_sS_h(t-s))P_hD_s\rho(s)ds-\int_{t/2}^tS_h(t-s)P_hD_{ss}\rho(s)ds.
\end{eqnarray}
From \eqref{essai3} it holds that  $S_h(t)\theta_t(0)+S_h(t)P_hD_t\rho(0)=S_h(t)A_h\theta(0)$. Therefore \eqref{essai6} becomes
\begin{eqnarray}
\label{essai7}
\theta_t(t)&=&S_h(t)A_h\theta(0)-S_h(t/2)P_hD_s\rho(t/2)+\int_0^{t/2}(D_sS_h(t-s))P_hD_s\rho(s)ds\nonumber\\
&-&\int_{t/2}^tS_h(t-s)P_hD_{ss}\rho(s)ds.
\end{eqnarray}
Taking the norm in both sides of \eqref{essai7},  using the stability properties of the semigroup and the boundedness of $P_h$ yields
\begin{eqnarray}
\label{essai8}
\Vert\theta_t(t)\Vert&\leq& Ct^{-1}\Vert \theta(0)\Vert+C\Vert D_t\rho(t/2)\Vert+\left\Vert\int_0^{t/2}(D_sS_h(t-s)) P_hD_s\rho(s) ds\right\Vert\nonumber\\
&+&C\int_{t/2}^t\Vert D_{ss}\rho(s)\Vert ds.
\end{eqnarray}
Using the definition of $\theta(t)$, the boundedness of $P_h$ and  \eqref{ritz1}, it holds that
\begin{eqnarray}
\label{essai9}
\Vert \theta (0)\Vert=\Vert P_hv-R_hv\Vert=\Vert P_hv-P_hR_hv\Vert\leq \Vert v-R_hv\Vert\leq Ch^{\alpha}\Vert v\Vert_{\alpha}.
\end{eqnarray}
Using \eqref{ritz1} and   \propref{prop1} as in \cite{Antonio1}, we obtain 
\begin{eqnarray}
\label{essai10}
\Vert D_t\rho(t)\Vert\leq Ch^{r}\Vert D_tu\Vert_{r}=Ch^{r}\Vert D_tS(t)v\Vert_{r}\leq Ch^{r}t^{-1-(r-\alpha)/2}\Vert v\Vert_{\alpha}.
\end{eqnarray}
Using again  \eqref{ritz1} and  Proposition \ref{prop1},  the following estimate  holds
\begin{eqnarray}
\label{essai11}
\Vert D_{ss}\rho(s)\Vert\leq Ch^{r}\Vert D_{ss}u\Vert_{r}=Ch^{r}\Vert D_{ss}S(s)v\Vert_{r}\leq Ch^{r}s^{-2-(r-\alpha)/2}\Vert v\Vert_{\alpha}.
\end{eqnarray}
Substituting \eqref{essai11}, \eqref{essai10} and \eqref{essai9} in \eqref{essai8}, it holds
\begin{eqnarray}
\label{essai12}
&&\Vert \theta_t(t)\Vert\nonumber\\
&\leq& C h^{\alpha}t^{-1}\Vert v\Vert_{\alpha}+\left\Vert\int_0^{t/2} (D_s(S_h(t-s))P_hD_s\rho(s) ds\right\Vert+Ch^{r}\int_{t/2}^ts^{-2-(r-\alpha)/2}\Vert v\Vert_{\alpha}ds\nonumber\\
&\leq&Ch^{\alpha}t^{-1}\Vert v\Vert_{\alpha}+\left\Vert\int_0^{t/2} (D_s(S_h(t-s))P_hD_s\rho(s)ds\right\Vert+Ch^{r}t^{-1-(r-\alpha)/2}\Vert v\Vert_{\alpha}.
\end{eqnarray}
 Using  the second relation of \eqref{essai3} and  triangle inequality, it holds that 
\begin{eqnarray}
\label{essai13}
&&\left\Vert\int_0^{t/2} \left(D_sS_h(t-s)\right)P_hD_s\rho(s) ds\right\Vert\nonumber\\
&=&\left\Vert\int_0^{t/2}\left(D_sS_h(t-s)\right)\left(A_h\theta(s)-\theta_s(s)\right) ds\right\Vert\nonumber\\
&\leq&\left\Vert\int_0^{t/2}\left(D^2_s(S_h(t-s)\right)\theta(s) ds\right\Vert+\left\Vert\int_0^{t/2} D_sS_h(t-s)\theta_s(s) ds\right\Vert.
\end{eqnarray}
Note that
\begin{eqnarray}
\label{essai13a}
D_sS_h(t-s)\theta_s(s)=D_s\left(D_sS_h(t-s)\theta(s)\right)-D^2_sS_h(t-s)\theta(s).
\end{eqnarray}
Substituting \eqref{essai13a} in \eqref{essai12} yields
\begin{eqnarray}
\label{essai13b}
&&\left\Vert\int_0^{t/2}D_sS_h(t-s)\theta_s(s)ds\right\Vert\nonumber\\
&\leq&\left\Vert\int_0^{t/2}D_s\left(A_hS_h(t-s)\theta(s)\right)ds\right\Vert+\int_0^{t/2}\left\Vert D^2_sS_h(t-s)\theta(s)\right\Vert ds\nonumber\\
&\leq& \Vert A_hS_h(t/2)\theta(t/2)\Vert+\Vert A_hS_h(t)\theta(0)\Vert\nonumber\\
&+&\int_0^{t/2}(t-s)^{-2}\Vert \theta(s)\Vert ds.
\end{eqnarray}
We recall that from \cite{Antonio1}, we have the following estimate
\begin{eqnarray}
\label{essai13c}
\Vert \theta(s)\Vert\leq Ch^rt^{-(r-\alpha)/2}\Vert v\Vert_{\alpha},\quad s>0.
\end{eqnarray}
Substituting \eqref{essai13c} and \eqref{essai9} in \eqref{essai13b} and using the smooth properties of the semigroup yields
\begin{eqnarray}
\label{essai13d}
&&\left\Vert \int_0^{t/2}D_sS_h(t-s)\theta(s)ds\right\Vert\nonumber\\
&\leq& Ch^rt^{-1-(r-\alpha)/2}\Vert v\Vert_{\alpha}+Ch^{\alpha}t^{-1}\Vert v\Vert_{\alpha}+C\int_0^{t/2}(t-s)^{-2}s^{-(r-\alpha)/2}\Vert v\Vert_{\alpha}ds\nonumber\\
&\leq& Ch^rt^{-1-(r-\alpha)/2}\Vert v\Vert_{\alpha}+Ch^{\alpha}t^{-1}\Vert v\Vert_{\alpha}+Ct^{-2}\int_0^{t/2}s^{-(r-\alpha)/2}\Vert v\Vert_{\alpha}ds\nonumber\\
&\leq& Ch^rt^{-1-(r-\alpha)/2}\Vert v\Vert_{\alpha}+Ch^{\alpha}t^{-1}\Vert v\Vert_{\alpha}.
\end{eqnarray}
Using \eqref{essai13c} and the stability property of the semigroup, it holds that
\begin{eqnarray}
\label{essai15}
\left\Vert \int_0^{t/2}\left(D_s^2S_h(t-s)\right)\theta(s)ds\right\Vert&\leq& \int_0^{t/2}\Vert D^2_sS_h(t-s)\theta(s)\Vert ds\nonumber\\
&\leq&C\int_0^{t/2}(t-s)^{-2}\Vert \theta(s)\Vert ds\nonumber\\
&\leq& C\int_0^{t/2}(t-s)^{-2}s^{-(r-\alpha)/2}\Vert v\Vert_{\alpha}ds\nonumber\\
&\leq& Ct^{-2}\int_0^{t/2}s^{-(r-\alpha)/2}\Vert v\Vert_{\alpha}ds\nonumber\\
&\leq& Ch^rt^{-1-(r-\alpha)/2}\Vert v\Vert_{\alpha}.
\end{eqnarray}
Substituting \eqref{essai15} and \eqref{essai13d} in \eqref{essai13} gives 
{\small
\begin{eqnarray}
\label{essai17}
\left\Vert \int_0^{t/2}(D_sS_h(t-s))P_hD_s\rho(s)ds\right\Vert&\leq& Ch^{\alpha}t^{-1}\Vert v\Vert_{\alpha}+ Ch^{r}t^{-1-(r-\alpha)/2}\Vert v\Vert_{\alpha}.
\end{eqnarray}
}
Substituting \eqref{essai17} in \eqref{essai12} gives 
\begin{eqnarray}
\label{essai18}
\Vert \theta_t(t)\Vert\leq Ch^{\alpha}t^{-1}\Vert v\Vert_{\alpha}+ Ch^{r}t^{-1-(r-\alpha)/2}\Vert v\Vert_{\alpha}.
\end{eqnarray}
Substituting \eqref{essai18} and \eqref{essai10} in \eqref{mil2}  gives
\begin{eqnarray}
\label{essai20}
\Vert D_t(u_h(t)-u(t))\Vert\leq Ch^{\alpha}t^{-1}\Vert v\Vert_{\alpha}+Ch^{r}t^{-1-(r-\alpha)/2}\Vert v\Vert_{\alpha}.
\end{eqnarray}
This completes the proof of (ii).

\item[(iii)] Applying \lemref{lemma1} with $r=\alpha=0$ shows that \eqref{mil3} holds  for $\rho=0$. 
Using the first relation of \eqref{essai3}, the stability property of the semigroup and \eqref{ritz1} with $r=1$ yields
\begin{eqnarray}
\label{esti1}
\Vert G_h(t)v\Vert&=&\Vert A_hS_h(t)A_h^{-1}P_hv-AS(t)A^{-1}v\Vert\nonumber\\
&\leq& \Vert A_hS_h(t)P_h(R_hA^{-1}v-A^{-1}v)\Vert+\Vert [A_hS_h(t)P_h-AS(t)]A^{-1}v\Vert\nonumber\\
&\leq& Ct^{-1}h\Vert A^{-1}v\Vert_1+C\Vert D_tG_h(t)A^{-1}v\Vert\nonumber\\
&=& Ct^{-1}h\Vert v\Vert_{-1}+C\Vert D_tG_h(t)A^{-1}v\Vert.
\end{eqnarray}
Applying Lemma \ref{lemma2} (ii) with $r=\alpha=1$ gives
\begin{eqnarray}
\label{esti1a}
\Vert D_tG_h(t)A^{-1}v\Vert\leq Cht^{-1}\Vert A^{-1}v\Vert_1= Cht^{-1}\Vert A^{-1/2}v\Vert=Cht^{-1}\Vert v\Vert_{-1}.
\end{eqnarray}
Substituting \eqref{esti1a} in \eqref{esti1} yields
\begin{eqnarray}
\Vert G_h(t)v\Vert\leq Cht^{-1}\Vert v\Vert_{-1}.
\end{eqnarray}
Therefore \eqref{mil3} holds for $\rho=1$. Hence the proof of (iii) is completed  by interpolation theory.
\item[(iv)] Note that 
\begin{eqnarray}
\label{monsi1}
 \int_0^tG_h(s)vds&=&\int_0^tA^{-1}AS(s)vds-
 \int_0^tA_h^{-1}A_hS_h(s)P_hvds\nonumber\\
 &=&A^{-1}(S(t)-\mathbf{I})v-A_h^{-1}(S_h(t)-\mathbf{I})P_hv.
\end{eqnarray}
Taking the norm in \eqref{monsi1} yields
{\small
\begin{eqnarray}
\label{mil5}
\left\Vert\int_0^tG_h(s)vds\right\Vert\leq \Vert(A_h^{-1}P_h-A^{-1})v\Vert+\Vert (S(t)A^{-1}-S_h(t)A_h^{-1}P_h)v\Vert.
\end{eqnarray}
}
Using the first relation of \eqref{essai3} and \eqref{ritz1} with $r=2-\rho$ yields
\begin{eqnarray}
\label{mil5a}
\Vert (A_h^{-1}P_h-A^{-1})v\Vert&=&\Vert (R_hA^{-1}-A^{-1})v\Vert=\Vert (R_h-\mathbf{I})A^{-1}v\Vert\nonumber\\
&\leq& Ch^{2-\rho}\Vert A^{-1}v\Vert_{2-\rho}= Ch^{2-\rho}\Vert A^{-\rho/2}v\Vert\nonumber\\
&\leq& Ch^{2-\rho}\Vert v\Vert_{-\rho}.
\end{eqnarray}
Using again the first relation of \eqref{essai3} and triangle inequality, it holds that
\begin{eqnarray}
\label{mil6}
\Vert S(t)A^{-1}v-S_h(t)A_h^{-1}P_hv\Vert&=&\Vert S(t)A^{-1}v-S_h(t)R_hA^{-1}v\Vert\nonumber\\
&\leq& \Vert S(t)A^{-1}v-S_h(t)P_hA^{-1}v\Vert\nonumber\\
&+&\Vert S_h(t)(P_hA^{-1}v-R_hA^{-1}v)\Vert.
\end{eqnarray}
Aplying \lemref{lemma1} with $r=\alpha=2-\rho$ yields
\begin{eqnarray}
\label{mil7}
\Vert S(t)A^{-1}v-S_h(t)P_hA^{-1}v\Vert\leq Ch^{2-\rho}\Vert A^{-1}v\Vert_{2-\rho}= Ch^{2-\rho}\Vert v\Vert_{-\rho}.
\end{eqnarray}
Using the boundedness of $S_h(t)$,  the triangle inequality,  the best approximation property of the orthogonal projector $P_h$ and  
the relation \eqref{ritz1} with $r=\alpha=2-\rho$, it holds that
\begin{eqnarray}
\label{mil8}
\Vert S_h(t)(P_hA^{-1}v-R_hA^{-1}v)\Vert&\leq& \Vert P_hA^{-1}v-R_hA^{-1}v\Vert\nonumber\\
&\leq& \Vert P_hA^{-1}v-A^{-1}v\Vert+\Vert A^{-1}v-R_hA^{-1}v\Vert\nonumber\\
&=&\Vert (P_h-\mathbf{I})A^{-1}v\Vert+\Vert (R_h-\mathbf{I})A^{-1}v\Vert\nonumber\\
&\leq &\Vert (R_h-\mathbf{I})A^{-1}v\Vert+\Vert (R_h-\mathbf{I})A^{-1}v\Vert\nonumber\\
&=& 2\Vert (R_h-\mathbf{I})A^{-1}v\Vert\nonumber\\
&\leq& Ch^{2-\rho}\Vert A^{-1}v\Vert_{2-\rho}\leq Ch^{2-\rho}\Vert v\Vert_{-\rho}.
\end{eqnarray}
Substituting \eqref{mil8} and \eqref{mil7} in \eqref{mil6} yields
\begin{eqnarray}
\label{mil9}
\Vert S(t)A^{-1}v-S_h(t)A_h^{-1}P_hv\Vert\leq Ch^{2-\rho}\Vert v\Vert_{-\rho}.
\end{eqnarray}
Substituting \eqref{mil9} and \eqref{mil5a} in \eqref{mil5} yields
\begin{eqnarray}
\label{mil10}
\left\Vert \int_0^tG_h(s)vds\right\Vert\leq Ch^{2-\rho}\Vert v\Vert_{-\rho}.
\end{eqnarray}
This completes the proof of (iv).
\item[(v)] The proof of (v) can be found in \cite[Lemma 6.1 (ii)]{Antonio2}. 
\end{itemize}
\end{proof}

Applying the  implicit Euler scheme to \eqref{deter1} gives the following fully discrete scheme for \eqref{model1a}
\begin{eqnarray}
\label{implicit2}
U^h_m+\Delta tA_hU^h_m=U^h_{m-1},\quad m=1,\cdots, M,\quad U^h_0=P_hv.
\end{eqnarray}
The numerical scheme \eqref{implicit2} can be written as follows:
\begin{eqnarray}
\label{implicit3}
U^h_m=S_{h, \Delta t}^mP_hv,\quad m=0,1\cdots, M,
\end{eqnarray}
where $S_{h, \Delta t}$ is defined by \eqref{operator1}.

\begin{lemma} 
\label{lemma3}
\begin{itemize}
\item[(i)] The backward difference \eqref{implicit3} is unconditionally stable. More precisely, for any $m$, $h$ and $\Delta t$ the following estimate holds
\begin{eqnarray*}
\Vert (\mathbf{I}+\Delta tA_h)^{-m}\Vert_{L(H)}\leq 1.
\end{eqnarray*}
\item[(ii)] Let  $0\leq q\leq 1$. Then for all $u \in H$, the following estimate holds
\begin{eqnarray*}
\Vert (S_h(t_m)-S_{h,\Delta t}^m)P_hu\Vert\leq C\Delta t^{q}t_m^{-q}\Vert u\Vert,\quad m=1,\cdots, M.
\end{eqnarray*}
\item[(iii)] For all $u\in\mathcal{D}(A^{\gamma-1})$, $0\leq\gamma\leq 2$, the following estimate holds
\begin{eqnarray*}
\Vert (S_{h,\Delta t}^m-S_h(t_m))P_hu\Vert\leq Ct_m^{-1/2}\Delta t^{\gamma/2}\Vert u\Vert_{\gamma-1}.
\end{eqnarray*}
\item[(iv)]
If $u\in\mathcal{D}(A^{\mu/2})$, $0\leq\mu\leq2$. Then the following error estimate holds 
\begin{eqnarray*}
\Vert (S_h(t_m)-S_{h,\Delta t}^m)P_hu\Vert\leq C\Delta t^{\mu/2}\Vert u\Vert_{\mu},\quad m=0,\cdots, M.
\end{eqnarray*}
\item[(v)] For  $u\in\mathcal{D}(A^{-\sigma})$, $0\leq\sigma\leq1$, the following estimate holds
\begin{eqnarray*}
\Vert (S_{h,\Delta t}^m-S_h(t_m))P_hu\Vert\leq Ct_m^{-1}\Delta t^{(2-\sigma)/2}\Vert u\Vert_{-\sigma}.
\end{eqnarray*}
\end{itemize}
\end{lemma}

\begin{proof}
\begin{itemize}
\item[(i)]
The proof of (i) can be found in \cite[Theorem 6.1]{Fujita}.

\item[(ii)]
Note that
\begin{eqnarray}
\label{monsi2}
\Vert (S_h(t_m)-S_{h,\Delta t}^m)P_hu\Vert\leq \Vert K_h(m)\Vert_{L(H)}\Vert P_hu\Vert,
\end{eqnarray}
where $K_h(m):=S_{h,\Delta t}^m-S_h(t_m)$. Then from \cite[(6.3)]{Fujita} and using the fact $\Delta t^{1-q}t_m^{-1+q}\leq 1$ we obtain
\begin{eqnarray}
\label{monsi3}
\Vert K_h(m)\Vert_{L(H)} \leq C\Delta tt_m^{-1}\leq C\Delta t^{q}t_m^{-q}.
\end{eqnarray}
Substituting \eqref{monsi3} in \eqref{monsi2} yields
\begin{eqnarray}
\label{tom2}
\Vert (S_h(t_m)-S_{h,\Delta t}^mP_h)u\Vert \leq C\Delta t^qt_m^{-q}\Vert u\Vert.
\end{eqnarray}
This  completes the proof of (ii).
\item[(iii)]
If $0\leq\gamma\leq 1$, then it follows from \cite[(70)]{Antonio2} (by taking $\alpha=\frac{1-\gamma}{2}$) that 
\begin{eqnarray}
\label{pass1}
\left\Vert A_h^{\frac{\gamma-1}{2}}P_hu\right\Vert\leq C\left\Vert A^{\frac{\gamma-1}{2}}u\right\Vert.
\end{eqnarray}
If $1\leq \gamma\leq 2$, then it follows from \cite[Lemma 1]{Antjd1} (by taking $\alpha=\frac{\gamma-1}{2}$) that 
\begin{eqnarray}
\label{pass2}
\left\Vert A_h^{\frac{\gamma-1}{2}}P_hu\right\Vert\leq C\left\Vert A^{\frac{\gamma-1}{2}}u\right\Vert.
\end{eqnarray}
Combining \eqref{pass1} and \eqref{pass2}, it holds that 
\begin{eqnarray}
\label{pass3}
\left\Vert A_h^{\frac{\gamma-1}{2}}P_hu\right\Vert\leq C\left\Vert A^{\frac{\gamma-1}{2}}u\right\Vert,\quad 0\leq \gamma\leq 2.
\end{eqnarray}
Inserting $A_h^{\frac{1-\gamma}{2}}A_h^{\frac{\gamma-1}{2}}$ and using \eqref{pass3} it holds that
\begin{eqnarray}
\label{debut0}
\left\Vert \left(S^m_{h,\Delta t}-S_h(t_m)\right)P_hu\right\Vert\leq \left\Vert K_h(m)A_h^{\frac{1-\gamma}{2}}\right\Vert_{L(H)}\Vert u\Vert_{\gamma-1}.
\end{eqnarray}
As  in \cite[(6.4)]{Fujita} we can easily check that
\begin{eqnarray}
\label{debut1}
-K_h(m)A_h^{\frac{1-\gamma}{2}}&=&\int_0^{\Delta t}\frac{d}{ds}\left((\mathbf{I}+sA_h)^{-m}e^{-m(\Delta t-s)A_h}\right)A_h^{\frac{1-\gamma}{2}}ds\nonumber\\
&=&m\int_0^{\Delta t}sA_h^2(\mathbf{I}+sA_h)^{-m-1}e^{-m(\Delta t-s)A_h}A_h^{\frac{1-\gamma}{2}}ds\nonumber\\
&=&m\int_0^{\Delta t}sA_h^{\frac{4-\gamma}{2}}(\mathbf{I}+sA_h)^{-(m+1)}e^{-m(\Delta t-s)A_h}A_h^{\frac{1}{2}}ds.
\end{eqnarray}
Using the stability property of the semigroup, it holds that
\begin{eqnarray}
\label{debut2}
\left\Vert e^{-m(\Delta t-s)A_h}A_h^{1/2}\right\Vert_{L(H)}\leq C(m(\Delta t-s))^{-1/2}.
\end{eqnarray}
Using \cite[(6.6)]{Fujita} with $\alpha=\frac{4-\gamma}{2}$, it holds that
\begin{eqnarray}
\label{debut3}
\left\Vert A_h^{\frac{4-\gamma}{2}}(\mathbf{I}+sA_h)^{-(m+1)}\right\Vert_{L(H)}\leq C((m+1)s)^{\frac{-4+\gamma}{2}}\leq  C(ms)^{\frac{-4+\gamma}{2}}.
\end{eqnarray}
Substituting \eqref{debut3} and \eqref{debut2} in \eqref{debut1} yields
\begin{eqnarray}
\label{debut4}
\left\Vert K_h(m)A_h^{\frac{1-\gamma}{2}}\right\Vert_{L(H)}&\leq&Cm\int_0^{\Delta t}s(ms)^{-2+\gamma/2}(m(\Delta t-s))^{-1/2}ds\nonumber\\
&\leq& Cm^{-3/2+\gamma/2}\int_0^{\Delta t}s^{-1+\gamma/2}(\Delta t-s)^{-1/2}ds\nonumber\\
&\leq& Cm^{-3/2+\gamma/2}\Delta t^{-1/2+\gamma/2}=Ct_m^{-3/2+\gamma/2}\Delta t=Ct_m^{-1/2}t_m^{-1+\gamma/2}\Delta t\nonumber\\
&\leq& Ct_m^{-1/2}t_1^{-1+\gamma/2}\Delta t= Ct_m^{-1/2}\Delta t^{\gamma/2}.
\end{eqnarray}
Substituting \eqref{debut4} in \eqref{debut0} completes the proof of (iii).
\item[(iv)]
Let us recall that for all $u\in\mathcal{D}(A^{\mu/2})$ the following estimate holds (see \cite[Lemma 1]{Antjd1})
\begin{eqnarray}
\label{pass4}
\left\Vert A_h^{\mu/2}P_hu\right\Vert\leq C\left\Vert A^{\mu/2}u\right\Vert, \quad 0\leq \mu \leq 2.
\end{eqnarray}
Inserting $A_h^{-\mu/2}A_h^{\mu/2}$ and using \eqref{pass4} yields
\begin{eqnarray}
\label{suite1}
\left\Vert \left(S_h(t_m)-S_{h,\Delta t}^m\right)P_hu\right\Vert&\leq& \left\Vert \left(S_h(t_m)-S_{h,\Delta t}^m\right)A_h^{-\mu/2}\right\Vert_{L(H)}\left\Vert A_h^{\mu/2}P_hu\right\Vert\nonumber\\
&\leq& C\left\Vert K_h(m)A_h^{-\mu/2}\right\Vert_{L(H)}\Vert u\Vert_{\mu}.
\end{eqnarray}
Following the same as the estimate \eqref{debut4}, one can easily prove that
\begin{eqnarray}
\label{suite5a}
\Vert K_h(m)A^{-\mu/2}\Vert_{L(H)}\leq C\Delta t^{\mu/2}.
\end{eqnarray}
Substituting \eqref{suite5a} in \eqref{suite1} yields
\begin{eqnarray}
\label{suite6}
\left\Vert \left(S_h(t_m)-S_{h,\Delta t}^m\right)P_hu\right\Vert\leq C\Delta t^{\mu/2}\Vert u\Vert_{\mu}.
\end{eqnarray}
This completes the proof of (iv).
\item[(v)]
Inserting $A_h^{-\sigma/2}A_h^{\sigma/2}$ and using  \cite[(70)]{Antonio2} with $\alpha=\sigma$, we obtain
\begin{eqnarray}
\label{fin1}
\left\Vert \left(S_{h,\Delta t}^m-S_h(t_m)\right)P_hu\right\Vert&=&\left\Vert \left(S_{h,\Delta t}^m-S_h(t_m)\right)A_h^{\sigma/2}A_h^{-\sigma/2}P_hv\right\Vert\nonumber\\
&\leq&\left\Vert K_h(m)A^{\sigma/2}\right\Vert_{L(H)} \Vert u\Vert_{-\sigma}.
\end{eqnarray}
Along the same lines as the estimate \eqref{debut4}, one can easily show that
\begin{eqnarray}
\label{fin5a}
\Vert K_h(m)A^{\sigma/2}\Vert_{L(H)}\leq Ct_m^{-1}\Delta t^{(2-\sigma)/2}.
\end{eqnarray}
Substituting \eqref{fin5a} in \eqref{fin1} completes the proof of (v).
\end{itemize}
\end{proof}

\begin{remark}
\lemref{lemma3} (ii) and (iv) generalise respectively \cite[Theorem 7.7]{Vidar} and \cite[Theorem 7.8]{Vidar} to general second-order
homogeneous parabolic equations.
\end{remark}

Let us introduce  the error  operator $G_{h,\Delta t}$ defined for all $t\in[0,T]$ by
\begin{eqnarray*}
G_{h,\Delta t}(t) := S_{h,\Delta t}^mP_h-S_h(t)P_h\quad\text{if}\quad t\in[t_{m-1}, t_m],\quad m=1,\cdots M.
\end{eqnarray*}

\begin{lemma}
\label{lemma4}
Let  $h, \Delta t\in(0,1]$. Then the following estimates hold 
\begin{itemize}
\item[(i)] Let $0\leq\rho\leq 1$. Then there exists a constant $C$ such that
\begin{eqnarray*}
\Vert G_{h,\Delta t}(t)u\Vert\leq Ct^{-\rho/2}\Vert u\Vert_{-\rho},\quad u\in\mathcal{D}(A^{-\rho}),\quad t>0.
\end{eqnarray*}
\item[(ii)] Let $0\leq\rho\leq 1$. Then there exists a constant $C$ such that
\begin{eqnarray}
\label{mes1}
\Vert G_{h,\Delta t}(t)u\Vert\leq C\Delta t^{\frac{(2-\rho)}{2}-\epsilon}t^{-1+\epsilon}\Vert u\Vert_{-\rho},\quad u\in\mathcal{D}(A^{-\rho}),
\end{eqnarray}
for all  $t>0$ and $\epsilon>0$ small enough.
\item[(iii)] Let $\mu\in[0,2]$ and $u\in\mathcal{D}(A^{\mu-1})$. Then the following estimate holds
\begin{eqnarray}
\label{jour0}
\Vert G_{h,\Delta t}(t)u\Vert\leq C\Delta t^{\frac{\mu}{2}-\epsilon}t^{-\frac{1}{2}+\epsilon}\Vert u\Vert_{\mu-1}.
\end{eqnarray}
\end{itemize}
\end{lemma}

\begin{proof}
\begin{itemize}
\item[(i)]
The proof of  (i) follows the sames lines as \cite[Lemma 4.3 (ii)]{Raphael}, which does not make use of the spectral decomposition and works for general linear operator.
\item[(ii)]
Using triangle inequality yields
\begin{eqnarray}
\label{mes2}
\Vert G_{h,\Delta t}(t)u\Vert&\leq& \Vert (S_{h,\Delta t}^m-S_h(t_m))P_hu\Vert+\Vert (S_h(t_m)-S_h(t_m))P_hu\Vert=:T_1+T_2.
\end{eqnarray}
Applying \lemref{lemma3} (v) with $\sigma=1$ yields 
\begin{eqnarray}
\label{mes3}
T_1\leq Ct_m^{-1}\Delta t^{(2-\rho)/2}\Vert u\Vert_{-\rho}\leq Ct^{-1}\Delta t^{\frac{(2-\rho)}{2}-\epsilon}\Vert u\Vert_{-\rho}.
\end{eqnarray}
Using the stability properties of the semigroup, it holds that
\begin{eqnarray}
\label{mes5}
T_2&=&\left\Vert S_h(t)(S_h(t_m-t)-\mathbf{I})A_h^{\rho/2}A_h^{-\rho/2}u\right\Vert\nonumber\\
&\leq& \left\Vert A_hS_h(t)(S_h(t_m-t)-\mathbf{I})A_h^{\rho/2}\right\Vert_{L(H)}\left\Vert A_h^{-\rho/2}u\right\Vert\nonumber\\
&\leq&\Vert A_hS_h(t)\Vert_{L(H)}\left\Vert(S_h(t_m-t)-\mathbf{I})A_h^{(\rho-2)/2}\right\Vert_{L(H)}\Vert u\Vert_{-\rho}\nonumber\\
&\leq & Ct^{-1}(t_m-t)^{(2-\rho)/2}\Vert u\Vert_{-\rho}\leq Ct^{-1}\Delta t^{(2-\rho)/2}\Vert u\Vert_{-\rho}.
\end{eqnarray}
Substituting \eqref{mes5}  and \eqref{mes3} in \eqref{mes2} proves \eqref{mes1} for $\rho=1$. The  proof of (ii) is completed by interpolation theory.
\item[(iii)] We have the following decomposition
\begin{eqnarray}
\label{jour1}
\Vert G_{h,\Delta t}(t)u\Vert&\leq& \Vert (S_{h,\Delta t}^m-S_h(t_m))P_hu\Vert+
\Vert (S_h(t_m)-S_h(t))P_hu\Vert.
\end{eqnarray}
Inserting $A^{\frac{1-\mu}{2}}A^{\frac{\mu-1}{2}}$ and  using the stability property of the semigroup, we obtain
\begin{eqnarray}
\label{jour2}
\Vert (S_h(t_m)-S_h(t))P_hu\Vert&\leq& \Vert\left(\mathbf{I}-S_h(t_m-t)\right)S_h(t)u\Vert\nonumber\\
&\leq&\left\Vert (\mathbf{I}-S_h(t_m-t))S_h(t)A_h^{\frac{1-\mu}{2}}A_h^{\frac{\mu-1}{2}}u\right\Vert\nonumber\\
&\leq& \left\Vert \left(\mathbf{I}-S_h(t_m-t)\right)A_h^{-\mu/2}S_h(t)A_h^{1/2}\right\Vert_{L(H)}\Vert u\Vert_{\mu-1}\nonumber\\
&\leq & \left\Vert (\mathbf{I}-S_h(t_m-t))A_h^{-\mu/2}\right\Vert_{L(H)}\left\Vert S_h(t)A_h^{1/2}\right\Vert_{L(H)}\Vert u\Vert_{\mu-1}\nonumber\\
&\leq&C(t_m-t)^{\mu/2}t^{-1/2}\Vert u\Vert_{\mu-1}\nonumber\\
&\leq& C\Delta t^{\mu/2}t^{-1/2}\Vert u\Vert_{\mu-1}.
\end{eqnarray}
Following closely \lemref{lemma3} (iii) one can show that
\begin{eqnarray}
\label{jour4}
\left\Vert \left(S^m_{h,\Delta t}-S_h(t_m)\right)P_hu\right\Vert\leq Ct_m^{-1/2}\Delta t^{\mu/2}\Vert u\Vert_{\mu-1}\leq C\Delta t^{\mu/2} t^{-1/2}\Vert u\Vert_{\mu-1}.
\end{eqnarray}
Substituting  \eqref{jour4} and \eqref{jour2}  in \eqref{jour1}  completes the proof of (iii).
\end{itemize}
\end{proof}

\begin{lemma}
\label{lemma5}
Let $0\leq\rho\leq 1$ and and $\epsilon>0$ an arbitrarily small number.
\begin{itemize}
\item[(i)] For all $u\in\mathcal{D}(A^{-\rho})$ there exists a positive constant $C$ such that
\begin{eqnarray*}
\left\Vert\sum_{j=1}^m\int_{t_{j-1}}^{t_j}\left(S_{h,\Delta t}^jP_h-S_h(s)P_h\right)uds\right\Vert\leq C\Delta t^{\frac{2-\rho}{2}-\epsilon}\Vert u\Vert_{-\rho}.
\end{eqnarray*}
\item[(ii)] For any $\gamma\in[0,2]$ and $u\in\mathcal{D}(A^{\gamma-1})$, the following estimate holds
\begin{eqnarray}
\label{an1}
\left(\sum_{j=1}^m\int_{t_{j-1}}^{t_j}\left\Vert \left(S_{h,\Delta t}^jP_h-S_h(s)P_h\right)u\right\Vert^2ds\right)^{1/2}\leq C\Delta t^{\frac{\gamma}{2}-\epsilon}\Vert u\Vert_{\gamma-1}.
\end{eqnarray}
\end{itemize}
\end{lemma}

\begin{proof}
\begin{itemize}
\item[(i)]
Using triangle inequality we obtain
\begin{eqnarray}
\label{mat1}
&&\left\Vert\sum_{j=1}^m\int_{t_{j-1}}^{t_j}\left(S_{h,\Delta t}^jP_h-S_h(s)P_h\right)uds\right\Vert\nonumber\\
&\leq&\left\Vert\int_{0}^{\Delta t}\left(S_{h,\Delta t}^1P_h-S_h(s)P_h\right)uds\right\Vert+\left\Vert\sum_{j=2}^m\int_{t_{j-1}}^{t_j}\left(S_{h,\Delta t}^jP_h-S_h(s)P_h\right)uds\right\Vert\nonumber\\
&:=&T_4+T_5.
\end{eqnarray}
Using again triangle inequality, we obtain
\begin{eqnarray}
\label{mat2}
T_4&\leq&\int_{0}^{\Delta t}\left\Vert \left(S_{h,\Delta t}^1-S_h(t_1)\right)P_hu\right\Vert ds+\int_0^{\Delta t}\left\Vert\left(S_h(t_1)P_h-S_h(s)P_h\right)u\right\Vert ds\nonumber\\
&:=&T_{41}+T_{42}.
\end{eqnarray}
Applying \lemref{lemma3} (v) with $\sigma=\rho$ yields
\begin{eqnarray}
\label{mat3}
T_{41}\leq C\int_0^{\Delta t}t_1^{-1}\Delta t^{(2-\rho)/2}\Vert u\Vert_{-\rho}ds\leq C\Delta t^{(2-\rho)/2}\Vert u\Vert_{-\rho}.
\end{eqnarray}
Using triangle inequality, inserting an appropriate power of $A_h$, using the stability properties of the semigroup and   \cite[(70)]{Antonio2}, it holds that
\begin{eqnarray}
\label{mat9a}
T_{42}&\leq &\int_{0}^{\Delta t}\left\Vert \left(S_h(t_1)-S_h(s)\right)P_hu\right\Vert ds\nonumber\\
&=& \int_{0}^{\Delta t}\left\Vert \left(S_h(t_1-s)-\mathbf{I}\right)S_h(s)A_h^{\rho/2}A_h^{-\rho/2}P_hu\right\Vert ds\nonumber\\
&\leq& C\int_{0}^{\Delta t}\left\Vert \left(S_h(t_1-s)-\mathbf{I}\right)A_h^{\frac{-2+\rho}{2}+\epsilon}\right\Vert_{L(H)}\Vert S_h(s)A_h^{1-\epsilon}\Vert_{L(H)}\Vert u\Vert_{-\rho}ds\nonumber\\
&\leq& C\int_{0}^{\Delta t}(\Delta t-s)^{\frac{(2-\rho)}{2}-\epsilon}s^{-1+\epsilon}\Vert u\Vert_{-\rho}ds\nonumber\\
&\leq& C\Delta t^{\frac{(2-\rho)}{2}-\epsilon}\Vert u\Vert_{-\rho}\int_{0}^{\Delta t}s^{-1+\epsilon}ds\nonumber\\
&\leq& C\Delta t^{\frac{(2-\rho)}{2}-\epsilon}\Vert u\Vert_{-\rho}.
\end{eqnarray}
Substituting \eqref{mat9a} and \eqref{mat3} in \eqref{mat2} yields
\begin{eqnarray}
\label{mat11}
T_{4}\leq C\Delta t^{\frac{2-\rho}{2}-\epsilon}\Vert u\Vert_{-\rho}.
\end{eqnarray}
Introducing the error operator $G_{h,\Delta t}$ in $T_5$, using triangle inequality and \lemref{lemma4} (ii),  we obtain
\begin{eqnarray}
\label{mat12}
T_5&=&\left\Vert\sum_{j=2}^m\int_{t_{j-1}}^{t_j}G_{h,\Delta t}(s)uds\right\Vert\leq \sum_{j=2}^m\int_{t_{j-1}}^{t_j}\Vert G_{h,\Delta t}(s)u\Vert ds\nonumber\\
&\leq& C\Delta t^{\frac{2-\rho}{2}-\epsilon}\Vert u\Vert_{-\rho}\sum_{j=2}^m\int_{t_{j-1}}^{t_j}s^{-1+\epsilon}ds=C\Delta t^{\frac{2-\rho}{2}-\epsilon}\Vert u\Vert_{-\rho}\int_{\Delta t}^{t_m}s^{-1+\epsilon}ds\nonumber\\
&\leq& C\Delta t^{\frac{2-\rho}{2}-\epsilon}\Vert u\Vert_{-\rho}.
\end{eqnarray}
Substituting \eqref{mat12} and \eqref{mat11} in \eqref{mat1} completes the proof of (i).
\item[(ii)]
 Using triangle inequality and the estimate $(a+b)^2\leq 2a^2+2b^2$, we obtain
 \begin{eqnarray}
 \label{vers1}
 &&\sum_{j=1}^m\int_{t_{j-1}}^{t_j}\left\Vert\left(S_{h,\Delta t}^jP_h-S_h(s)P_h\right)u\right\Vert^2ds\nonumber\\
 &\leq&2\int_{0}^{\Delta t}\left\Vert\left(S_{h,\Delta t}^1-S_h(s)\right)P_hu\right\Vert^2ds+2\sum_{j=2}^m\int_{t_{j-1}}^{t_j}\left\Vert\left(S_{h,\Delta t}^jP_h-S_h(s)P_h\right)u\right\Vert^2ds\nonumber\\
 &:=&2T_6+2T_7.
 \end{eqnarray}
 Using triangle inequality and the estimate $(a+b)^2\leq 2a^2+2b^2$, we obtain
 \begin{eqnarray}
 \label{vers2}
 T_6&\leq& 2\int_0^{\Delta t}\left\Vert \left(S^1_{h,\Delta t}P_h-S_h(t_1)\right)P_hu\right\Vert^2ds+2\int_0^{\Delta t}\left\Vert\left(S_h(t_1)P_h-S_h(s)P_h\right)u\right\Vert^2ds\nonumber\\
 &:=&2T_{61}+2T_{62}.
 \end{eqnarray}
Using \lemref{lemma3} (iii) yields 
\begin{eqnarray}
\label{vers3}
T_{61}&\leq& C\int_{0}^{\Delta t}t_1^{-1}\Delta t^{\gamma}\Vert u\Vert^2_{\gamma-1}ds\leq C\Delta t^{\gamma}\Vert u\Vert^2_{\gamma-1}.
\end{eqnarray}
Inserting an appropriate power of $A_h$, using \eqref{pass3} and the stability properties of the semigroup yields
\begin{eqnarray}
\label{vers6a}
T_{62} &\le& \int_{0}^{\Delta t}\left\Vert \left(S_h(t_1-s)-\mathbf{I}\right)S_h(s)A_h^{(1-\gamma)/2}\right\Vert^2_{L(H)}\left\Vert A_h^{(\gamma-1)/2}P_hu\right\Vert^2ds\nonumber\\
&\leq& C\int_{0}^{\Delta t}\left\Vert \left(S_h(t_1-s)-\mathbf{I}\right)A_h^{-\frac{\gamma}{2}+\epsilon}\right\Vert^2_{L(H)}\left\Vert S_h(s)A_h^{\frac{1}{2}-\epsilon}\right\Vert^2_{L(H)}\Vert u\Vert^2_{\gamma-1}ds\nonumber\\
&\leq& C\int_{0}^{\Delta t}(t_1-s)^{\gamma-2\epsilon}\left\Vert S_h(s)A_h^{\frac{1}{2}-\epsilon}\right\Vert^2_{L(H)}\Vert u\Vert^2_{\gamma-1}ds\nonumber\\
&\leq& C\Delta t^{\gamma-2\epsilon}\Vert u\Vert_{\gamma-1}^2\int_{0}^{\Delta t}s^{-1+2\epsilon}ds\nonumber\\
&\leq& C\Delta t^{\gamma-2\epsilon}\Vert u\Vert_{\gamma-1}^2.
\end{eqnarray}
Substituting \eqref{vers6a} and \eqref{vers3} in \eqref{vers2} yields
\begin{eqnarray}
\label{vers8}
T_6\leq C\Delta t^{\gamma-2\epsilon}\Vert u\Vert^2_{\gamma-1}.
\end{eqnarray}
Introducing the error operator $G_{h,\Delta t}$ in $T_7$ and using \lemref{lemma4} (iii) yields
\begin{eqnarray}
\label{vers9}
T_7&=&\sum_{j=2}^m\int_{t_{j-1}}^{t_j}\Vert G_{h,\Delta t}(s)u\Vert^2ds= C\Delta t^{\gamma-2\epsilon}\Vert u\Vert^2_{\gamma-1}\sum_{j=2}^m\int_{t_{j-1}}^{t_j}s^{-1+2\epsilon}ds\nonumber\\
&=&C\Delta t^{\gamma-2\epsilon}\Vert u\Vert^2_{\gamma-1}\int_{\Delta t}^{t_m}s^{-1+2\epsilon}ds
\leq C\Delta t^{\gamma-2\epsilon}\Vert u\Vert^2_{\gamma-1}.
\end{eqnarray}
Substituting \eqref{vers9} and \eqref{vers8} in \eqref{vers1} completes the proof of (ii).
\end{itemize}
\end{proof}

\begin{remark}
The above error estimate for deterministic problem can also be used to extend the result in \cite{Kruse} to the case of not necessarily self-adjoint operator.
\end{remark}

\begin{lemma}\textbf{[Space error]}
\label{spaceerror}
Let  \assref{assumption2}, \assref{assumption3}, \assref{assumption4} and \assref{assumption5} be fulfilled, 
then the following error estimate holds for the mild solution \eqref{mild1} and the mild solution of the semidiscrete problem \eqref{semi1}
\begin{eqnarray}
\label{espa}
\Vert X(t)-X^h(t)\Vert_{L^p(\Omega,H)}\leq Ch^{\beta},\quad t\in[0,T].
\end{eqnarray}
Moreover, for additive noise, if we replace  \assref{assumption5} by \assref{assumption6b}, then the error estimate \eqref{espa}  holds.
\end{lemma}

\begin{proof}
The proof follows the same lines as \cite[Theorem 1.1]{Raphael}. Note that the proof of \cite[Theorem 1.1]{Raphael} uses many preparatory 
results which are heavily based on the spectral decomposition of the linear operator. This is to achieve optimal convergence order without
any logarithmic perturbation. With our above preparatory results, the proof of \lemref{spaceerror} is straightforward. \lemref{spaceerror} 
improves  \cite[Lemma 8]{Antjd1} in the case $\beta=2$. Although the case of not necessary self-adjoint operator were analysed 
in \cite[Theorem 6.1]{Antonio2}, the crucial case $\beta=2$ were not studied. 
\end{proof}

\subsection{Proof of \thmref{mainresult1} and \thmref{mainresult2}}
\label{mainproof}
With the above  new preparatory results, the proof of \thmref{mainresult1} follows the same lines as that 
of \cite[Theorem 1.2]{Raphael}. The proof of \thmref{mainresult2} follows the same lines as \cite[Theorem 4.1]{Xiaojie2}.
As mentioned in the introduction, the proof of \cite[Theorem 1.2]{Raphael} and \cite[Theorem 4.1]{Xiaojie2} 
use many preparatory results which are heavily based on the spectral decomposition of the linear self-adjoint operator $A$.

\section{Numerical simulations}
\label{numexperiment}
We consider the stochastic dominated advection
diffusion reaction SPDE \eqref{model} with  constant diagonal difussion tensor  $ \mathbf{D}= 10^{-2} \mathbf{I}_2=(D_{i,j}) $ in \eqref{operator},
and mixed Neumann-Dirichlet boundary conditions on $\Lambda=[0,L_1]\times[0,L_2]$. 
The Dirichlet boundary condition is $X=1$ at $\Gamma=\{ (x,y) :\; x =0\}$ and 
we use the homogeneous Neumann boundary conditions elsewhere.
The eigenfunctions $ \{e_{i,j} \} =\{e_{i}^{(1)}\otimes e_{j}^{(2)}\}_{i,j\geq 0}
$ of  the covariance operator $Q$ are the same as for  Laplace operator $-\varDelta$  with homogeneous boundary condition,  given by 
\begin{eqnarray*}
e_{0}^{(l)}(x)=\sqrt{\dfrac{1}{L_{l}}},\qquad 
e_{i}^{(l)}(x)=\sqrt{\dfrac{2}{L_{l}}}\cos\left(\dfrac{i \pi }{L_{l}} x\right),
\, i=\mathbb{N}
\end{eqnarray*}
where $l \in \left\lbrace 1, 2 \right\rbrace,\, x\in \Lambda$.
We assume that the noise can be represented as 
 \begin{eqnarray}
  \label{eq:W1}
  W(x,t)=\underset{i \in  \mathbb{N}^{2}}{\sum}\sqrt{\lambda_{i,j}}e_{i,j}(x)\beta_{i,j}(t), 
\end{eqnarray}
where $\beta_{i,j}(t)$ are
independent and identically distributed standard Brownian motions,  $\lambda_{i,j}$, $(i,j)\in \mathbb{N}^{2}$ are the eigenvalues  of $Q$, with
\begin{eqnarray}
\label{noise2}
 \lambda_{i,j}=\left( i^{2}+j^{2}\right)^{-(\beta +\delta)}, \, \beta>0,
\end{eqnarray} 
in the representation \eqref{eq:W1} for some small $\delta>0$. For additive noise, we take $B(u)=2$, so \assref{assumption6a} is obviously satisfied for $\beta=(0,2]$. 
For multiplicative noise, we take $b(u)=2u$ in \eqref{nemystskii}. Therefore, from \cite[Section 4]{Arnulf1} it follows that the operators  $B$ defined by \eqref{nemystskii}
fulfills obviously \assref{assumption4} and \assref{assumption5}.
For both additive and multiplicative noise, we consider the function $f$ used in \eqref{nemystskii} to be  $f(x,z)= \frac{z}{1+z^2}$ for all $(x,z)\in \Lambda\times\mathbb{R}$. Therefore the corresponding Nemytskii operator defined by \eqref{nemystskii} is given by 
$F(v)= \dfrac{v}{1+v^2}$, $v\in H$. We can easily check that 
\begin{eqnarray}
\label{deri1}
\vert f(x,z)\vert\leq C(1+\vert z\vert), \quad \frac{\partial f}{\partial z}(x,z)=\frac{1-z^2}{(1+z^2)^2},\quad \left\vert\frac{\partial f}{\partial z}(x,z)\right\vert\leq C,\quad (x,z)\in\Lambda\times\mathbb{R},\\
\label{deri2}
\frac{\partial^2f}{\partial z^2}(x,z)=\frac{-6z+2z^3}{(1+z^2)^3},\quad\frac{\partial^2 f}{\partial x_i\partial z}(x,z)=0,\quad \left\vert \frac{\partial^2 f}{\partial x_i\partial z}(x,z)\right\vert\leq C,\quad (x,z)\in\Lambda\times\mathbb{R}.
\end{eqnarray}
Classical estimates yield
\begin{eqnarray}
\label{deri2a}
\left\vert\frac{\partial^2f}{\partial z^2}(x,z)\right\vert\leq \frac{6\vert z\vert}{(1+\vert z\vert^2)^3}+\frac{2\vert z\vert^3}{(1+\vert z\vert^2)^3}\leq 8,\quad x\in\Lambda,\quad z\in\mathbb{R}.
\end{eqnarray}
One can also easily check  that the Nemytskii operator $F(v)(x)=f(x,v(x))$, $x\in\Lambda$, $v\in H$ is twice differentiable, and the derivatives are given by
\begin{eqnarray}
\label{deri3}
F'(v)(u)(x)=\frac{\partial f}{\partial z}(x, v(x))u(x),\quad x\in\Lambda,\quad u,v\in H,\\
\label{deri4}
F''(v)(u_1,u_2)(x)=\frac{\partial^2f}{\partial z^2}(x,v(x))u_1(x)u_2(x),\quad x\in\Lambda,\quad u_1,u_2,v\in H.
\end{eqnarray}
Using \eqref{deri1} and  \eqref{deri3} it holds that
\begin{eqnarray}
\Vert F'(v)u\Vert&\leq& C\Vert u\Vert,\quad u,v\in H.
\end{eqnarray}
Using Cauchy-Schwartz inequality and \eqref{deri2a}, it follows from \eqref{deri4} that
\begin{eqnarray}
\label{deri5}
\Vert F''(v)(u_1, u_2)\Vert_{L^1(\Lambda,H)}&=&\int_{\Lambda}\vert F''(v)(u_1,u_2)(x)\vert dx\nonumber\\
&\leq& C\int_{\Lambda}\vert u_1(x)\vert\vert u_2(x)\vert dx\leq C\Vert u_1\Vert\Vert u_2\Vert.
\end{eqnarray}
Employing \eqref{deri5}, it holds that
\begin{eqnarray}
\label{nem3}
\left\Vert A^{-\eta}F''(v)(u_1, u_2)\right\Vert
&=&\sup_{\Vert w\Vert\leq 1}\left\vert\left\langle A^{-\eta}F''(v)(u_1, u_2), w\right\rangle\right\vert\nonumber\\
&=&\sup_{\Vert w\Vert\leq 1}\left\vert\left\langle F''(v)(u_1, u_2), (A^*)^{-\eta}w\right\rangle\right\vert\nonumber\\
&\leq& \left\Vert F''(v)(u_1, u_2)\right\Vert_{L^1(\Lambda, \mathbb{R})}\times\sup_{\Vert w\Vert\leq 1}\left\Vert (A^*)^{-\eta}w\right\Vert_{L^{\infty}(\Lambda,\mathbb{R})}\nonumber\\
&\leq& C\Vert v_1\Vert\Vert v_2\Vert \sup_{\Vert w\Vert\leq 1}\left\Vert (A^*)^{-\eta}w\right\Vert_{L^{\infty}(\Lambda,\mathbb{R})}.
\end{eqnarray}
Using the Sobolev embedding 
\begin{eqnarray}
\label{nem4}
H^{2\eta}(\Lambda)\hookrightarrow L^{\infty}(\Lambda, \mathbb{R}),\quad \eta> \frac{d}{4},\; d=1, 2, 3,
\end{eqnarray}
it holds that
\begin{eqnarray}
\label{nem5}
\left\Vert (A^*)^{-\eta}w\right\Vert_{L^{\infty}(\Lambda, \mathbb{R})}&\leq& C\left\Vert (A^*)^{-\eta}w\right\Vert_{H^{2\eta}}\leq C\left\Vert (A)^{\eta}(A^*)^{-\eta}w\right\Vert\nonumber\\
 &\leq&C\left\Vert (A)^{\eta}(A^*)^{-\eta}\right\Vert_{L(H)}\Vert w\Vert\leq C\Vert w\Vert,
\end{eqnarray}
where at the last step we have used \cite[Lemma 3.1]{Antonio2}. Substituting \eqref{nem5} in \eqref{nem3} yields
\begin{eqnarray*}
\left\Vert(A)^{-\eta}F''(v)(u_1,u_2)\right\Vert\leq C\Vert u_1\Vert\Vert u_2\Vert.
\end{eqnarray*}
Therefore Assumptions \ref{assumption3} and \ref{assumption6b} are fulfilled.
We obtain the Darcy velocity field $\mathbf{q}=(q_i)$  by solving the following  system
\begin{equation}
  \label{couple1}
  \nabla \cdot\mathbf{q} =0, \qquad \mathbf{q}=-\mathbf{k} \nabla p,
\end{equation}
with  Dirichlet boundary conditions on 
$\Gamma_{D}^{1}=\left\lbrace 0,L_1 \right\rbrace \times \left[
  0,L_2\right] $ and Neumann boundary conditions on
$\Gamma_{N}^{1}=\left( 0,L_1\right)\times\left\lbrace 0,L_2\right\rbrace $ such that 
\begin{eqnarray*}
p&=&\left\lbrace \begin{array}{l}
1 \quad \text{in}\quad \left\lbrace 0 \right\rbrace \times\left[ 0,L_2\right]\\
0 \quad \text{in}\quad \left\lbrace L_1 \right\rbrace \times\left[ 0,L_2\right]
\end{array}\right. 
\end{eqnarray*}
and $- \mathbf{k} \,\nabla p (\mathbf{x},t)\,\cdot \mathbf{n} =0$ in  $\Gamma_{N}^{1}$. 
Note that $\mathbf{k}$ is the permeability tensor.  We use a random permeability field as in \cite[Figure 6]{Antofirst}.
Note that  in the coercivity inequality \eqref{ellip1}, we have $c_0=0$,  and that  keeping  the advection term in the nonlinear function $F$ as in \cite{Raphael} was
very unstable in our numerical experiments.
 The permeability field and the streamline of the velocity field $\mathbf{q}$ are given in \figref{FIGI}(b) and \figref{FIGI}(c) respectively.
To deal with high  P\'{e}clet number,  we discretise in space using
finite volume method, viewed as a finite element method (see \cite{Antonio3}).
We take $L_1= 3$ and $L_2=2$ and 
our reference solutions samples are numerical solutions using at time step of $\Delta t = 1/ 2048$. The errors are computed at the final time $T=1$.
The initial  solution  is $X_0=0$, so we can therefore expect high orders convergence, which depend  only on the noise term. For both additive and multiplicative noise, we use
$\beta =2$ and $\delta=10^{-3}$.
In \figref{FIGI}(a),  the order of convergence is $0.55$ for multiplicative noise and  $ 0.95$ for additive noise, which are close to $0.5$ and $1$ 
in our theoretical results in \thmref{mainresult1} and \thmref{mainresult2} respectively. The mean of the solution is given in \figref{FIGI}(d).
\begin{figure}[!ht]
  \subfigure[]{
    \label{FIGIa}
    \includegraphics[width=0.47\textwidth]{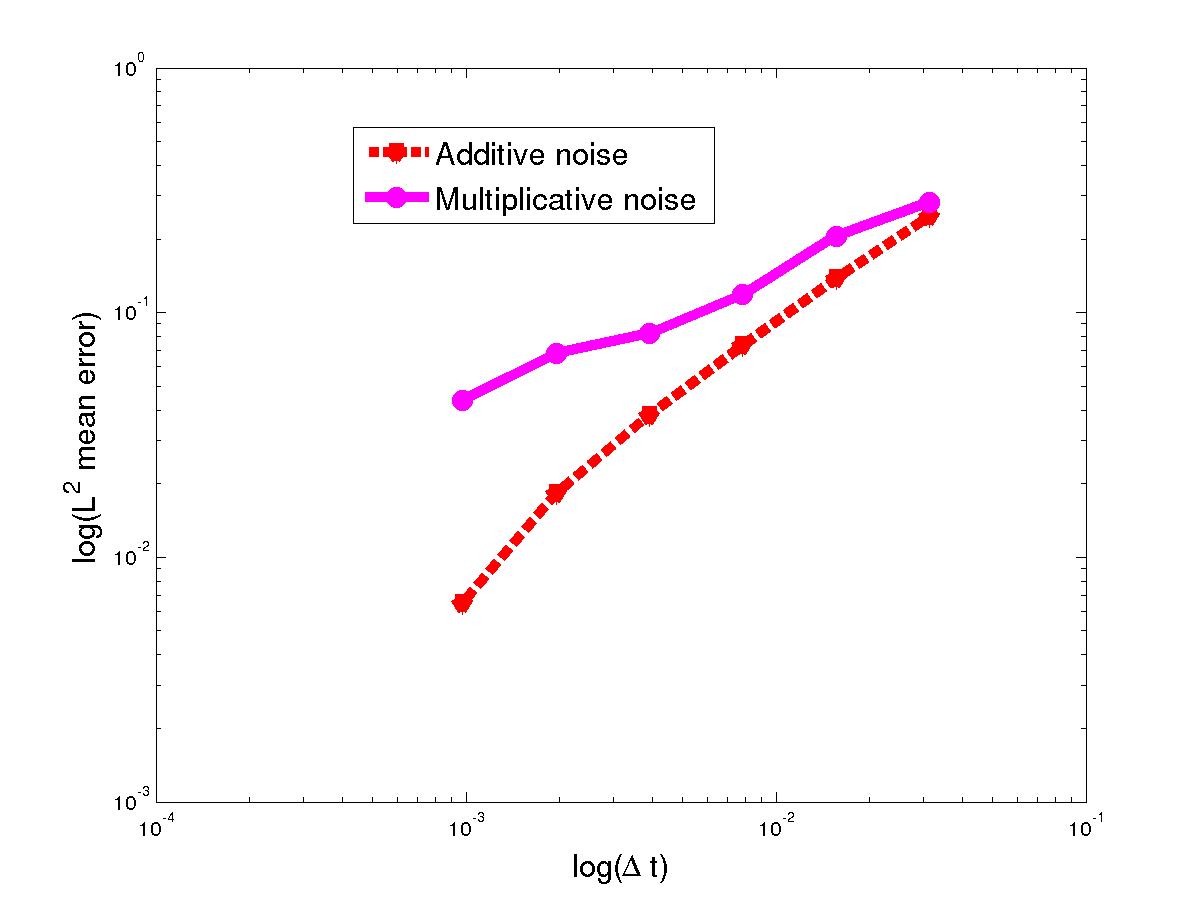}}
  \subfigure[]{
    \label{FIGIb}
    \includegraphics[width=0.47\textwidth]{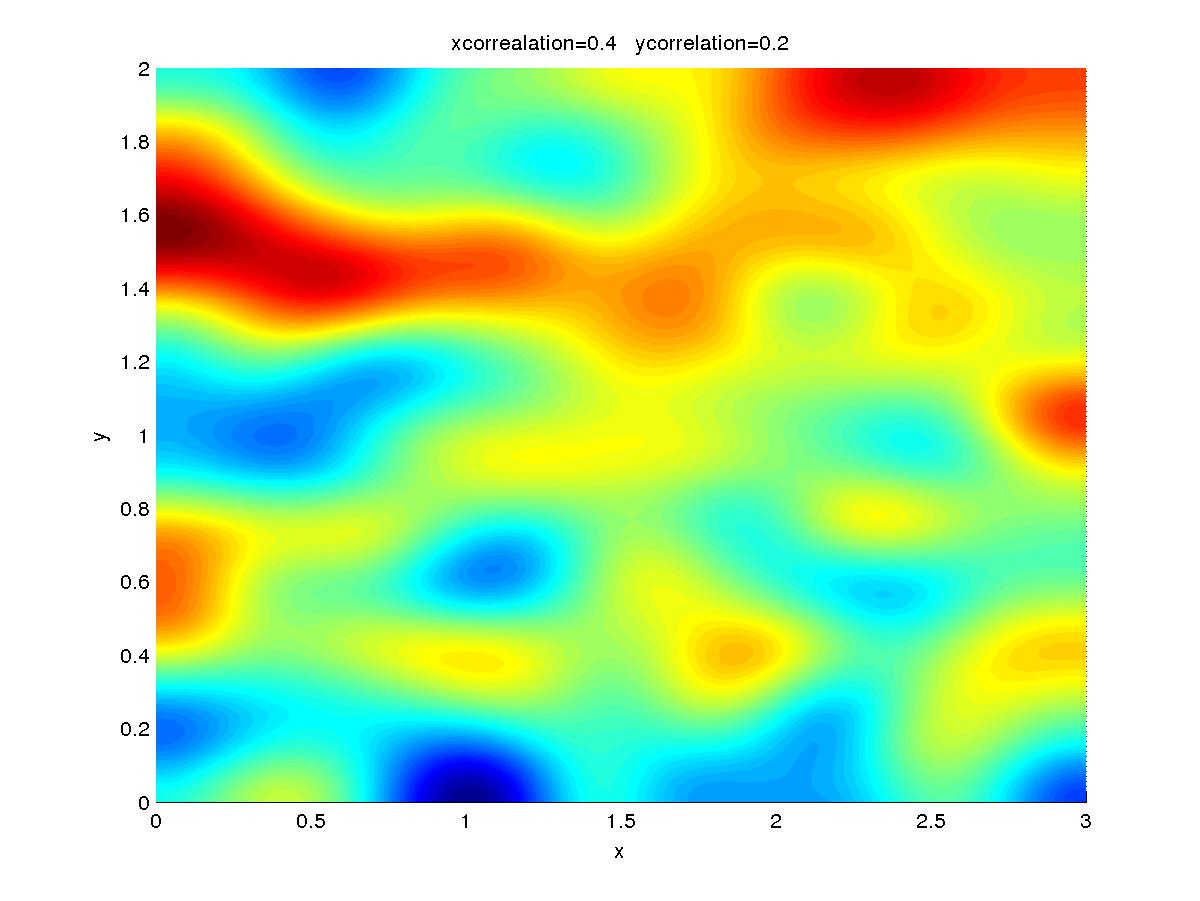}}
    \subfigure[]{
    \label{FIGIc}
    \includegraphics[width=0.47\textwidth]{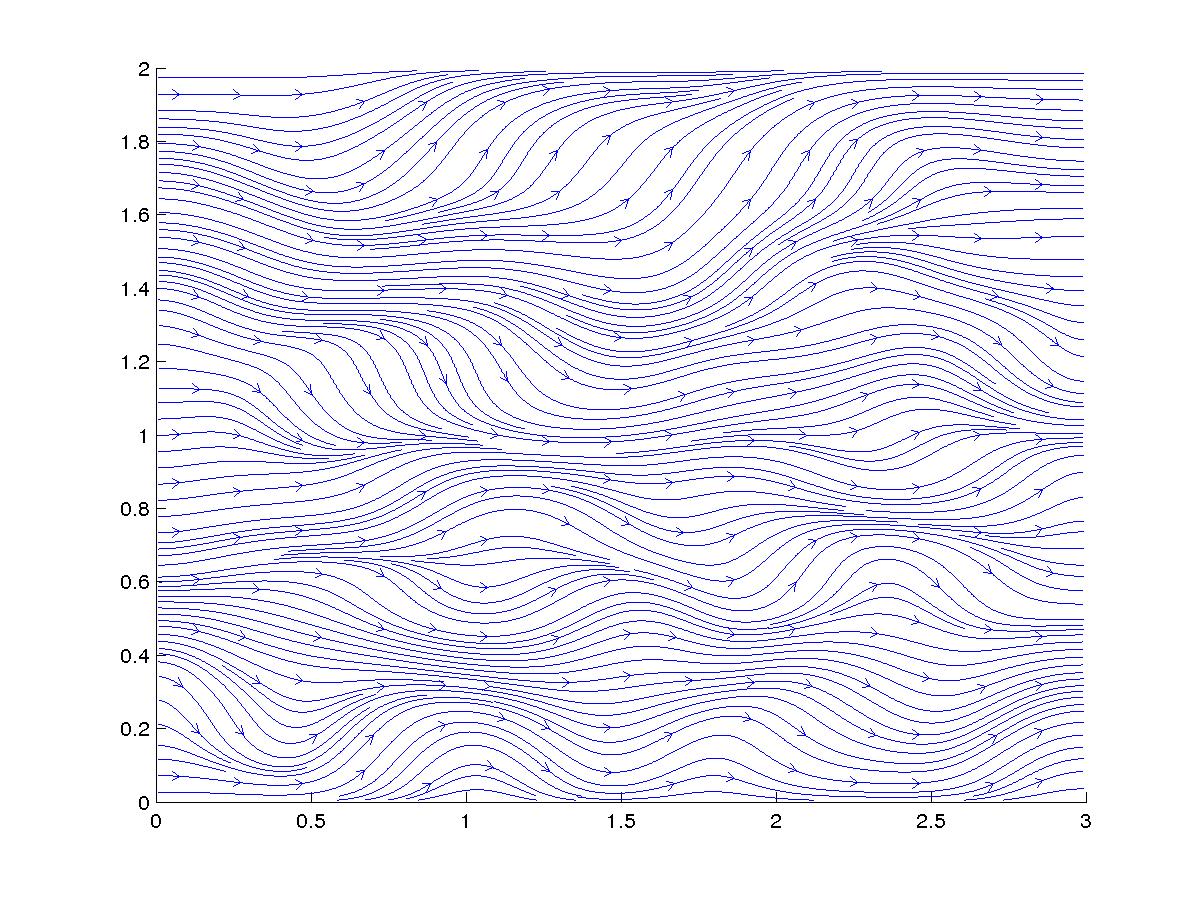}}
   \subfigure[]{
   \label{FIGId}
    \includegraphics[width=0.47\textwidth]{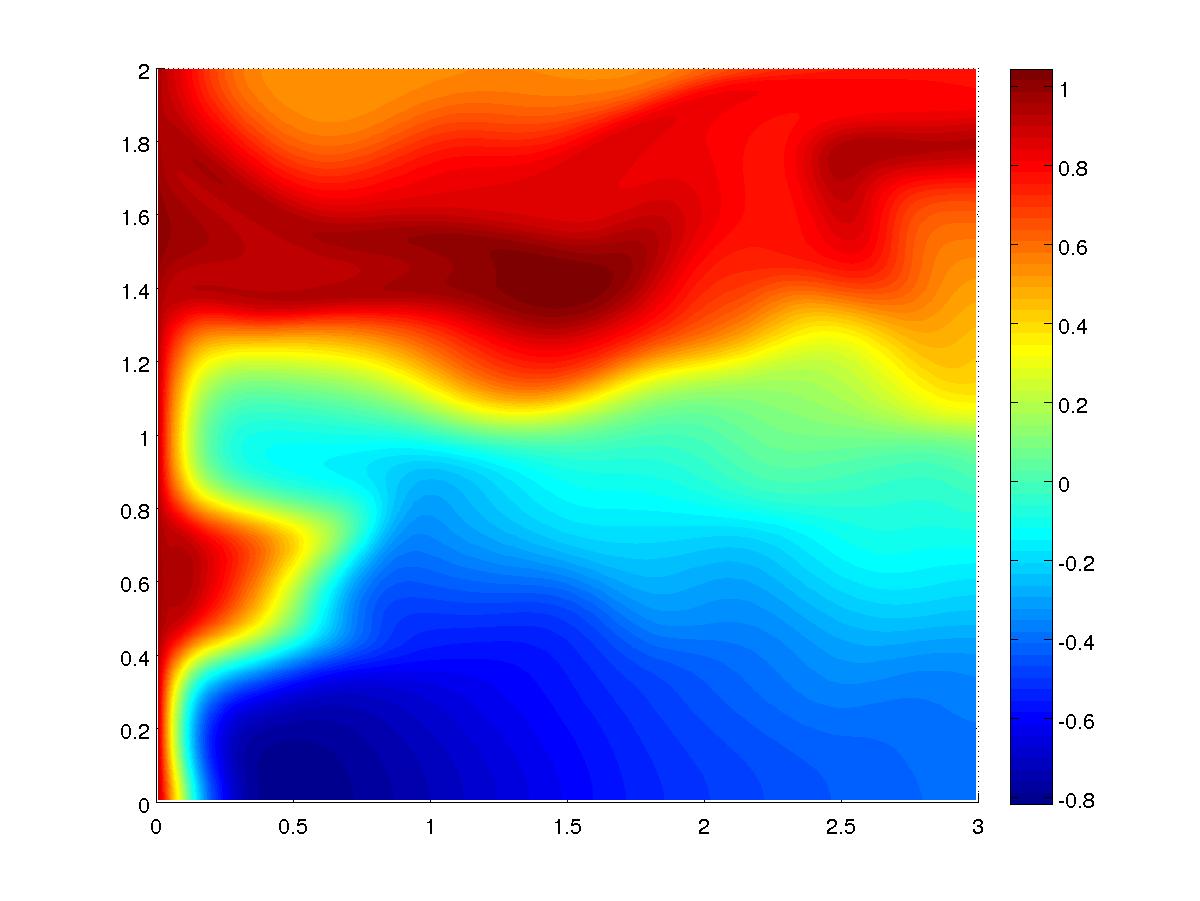}}
  \caption{(a) Convergence in the root mean square $L^{2}$ norm at $T=1$ as a
    function of $\Dt$. We show convergence for noise  where 
    $\beta=2$, and  $\delta=10^{-3}$ in relation \eqref{noise2}. We have  used here 30 realizations.  The order of convergence is $0.55$ for multiplicative noise and $0.95$ for additive noise. 
    Graph (b)  is the permeability field and  graph (c) is  the streamline of the velocity field $\mathbf{q}$.  The mean of the solution is given in (d).} 
  \label{FIGI} 
  \end{figure}
\section*{Acknowledgement}

A. Tambue was supported by  the Robert Bosch Stiftung through the AIMS ARETE Chair programme (Grant No 11.5.8040.0033.0). J. D. Mukam acknowledges  
the financial support of the TU Chemnitz and thanks Prof. Dr. Peter Stollmann for his constant support. We would like to thank Dr.  Raphael Kruse for very useful discussions at an
early stage of this paper.

\end{document}